\documentclass{amsproc} 
\usepackage{euscript}
\usepackage{cases}
\usepackage{mathrsfs}
\usepackage{bbm}
\usepackage{amssymb}
\usepackage{amsfonts,amsmath,amsxtra,mathdots,mathabx}
\usepackage{color}
\usepackage{hyperref}
\usepackage{tikz}
\usepackage{appendix,upgreek,setspace}

\allowdisplaybreaks

\DeclareFontFamily{U}{matha}{\hyphenchar\font45}
\DeclareFontShape{U}{matha}{m}{n}{
	<5> <6> <7> <8> <9> <10> gen * matha
	<10.95> matha10 <12> <14.4> <17.28> <20.74> <24.88> matha12
}{}
\DeclareSymbolFont{matha}{U}{matha}{m}{n}

\DeclareMathSymbol{\Lt}{3}{matha}{"CE}
\DeclareMathSymbol{\Gt}{3}{matha}{"CF}

\DeclareSymbolFont{mathc}{OML}{txmi}{m}{it}% txfonts
\DeclareMathSymbol{\varuu}{\mathord}{mathc}{117}
\DeclareMathSymbol{\varvv}{\mathord}{mathc}{118}
\DeclareMathSymbol{\varww}{\mathord}{mathc}{119}

\def\VA{\Vert \boldsymbol{a}_Y \Vert }

\def\BSD{   {\EuScript{D}} }

\def\valpha{\text{\scalebox{0.84}[1.02]{$\alpha$}}}   
\def\vepsilon{\upvarepsilon}
\def\vnu{\text{{\scalebox{0.9}[1]{$\nu$}}}} 
\def\vkappa{\text{{\scalebox{0.86}[1.1]{$\kappa$}}}} 
\def\vvkappa{\text{{\scalebox{0.8}[1.1]{$\varkappa$}}}} 

\def\bfF{\mathbf{F}}
\def\bfQ{\mathbf{Q}}
\def\bfK{\mathbf{K}}

\newcommand{\As}{{\mathrm {As}}}

\def\SO {\text{\raisebox{- 2 \depth}{\scalebox{1.1}{$ \text{\usefont{U}{BOONDOX-calo}{m}{n}O}  $}}}}
\def\RU {\mathrm{U}}

\newcommand{\BC}{{\mathbf {C}}}
\newcommand{\BQ}{{\mathbf {Q}}} 
\newcommand{\BR}{{\mathbf {R}}} 
\newcommand{\BZ}{{\mathbf {Z}}}

\def\Tr{ \mathit{Tr}}

\newcommand{\GL}{{\mathrm {GL}}}

\newcommand{\SL}{{\mathrm {SL}}}

\newcommand{\ra}{\rightarrow} 
\def\sumx{\sideset{}{^\star}\sum}

\def\sumn{\sideset{}{^\natural}\sum}
\def\mod{\mathrm{mod} }
\def\nd{\mathrm{d}}

\def\lp {\left (}
\def\rp {\right )}

\newcommand{\delete}[1]{}

\theoremstyle{plain}

\newtheorem{coro}{Corollary}[section]
\newtheorem{lem}{Lemma}[section]
\newtheorem{theorem}{Theorem}[section] 
\newtheorem{proposition}{Proposition}[section]

\newtheorem*{thm*}{Theorem}

\theoremstyle{remark} 
\newtheorem{remark}{Remark}[section] 

\numberwithin{equation}{section}

\begin{document}
	
		\title{On the Second Moment of $ L (1/2, \mathrm{As}(f) \times \phi)$}

	\begin{abstract}
		Let $\mathbf{F} = \mathbf{Q}(\sqrt D)$ be a real quadratic field. In this paper, we establish a large sieve inequality for the Asai lifts $ \mathrm{As} (f) $ with $f$ in a Hecke orthonormal basis of the space of Hilbert modular cusp forms of parallel weight $(k, k)$ over $ \mathbf{F} $. As an application, for a fixed Hecke--Maass cusp form $\phi $ over $\mathbf{Q}$, we prove a non-trivial bound for the second moment of the convoluted central $L$-values $ L (1/2, \mathrm{As}(f) \times \phi) $ in the $k$-aspect. 
	\end{abstract}
	
\author[C. Li and Z. Qi]{Changlin Li  and Zhi Qi}
\address{School of Mathematical Sciences\\ Zhejiang University\\Hangzhou, 310027\\China}
\email{12135012@zju.edu.cn, zhi.qi@zju.edu.cn} 
	
	\thanks{The second author was supported by National Key R\&D Program of China No. 2022YFA1005300.}

	
	\maketitle
	
	 \begin{spacing}{1.3}
{  \tableofcontents}	
 \end{spacing}

	\section{Introduction}\label{secL intro}
	
	Let $\bfF = \bfQ(\sqrt D)$ be a fixed real quadratic field, with $D > 1  $ square-free. Let $\SO $, $\SO^+$, and $\RU$ denote the ring of integers, the set of totally positive integers, and  the group of units, respectively. For simplicity, assume the narrow class number $h_{\bfF} ^{+} = 1$ so that the totally positive units are squares of units and every ideal has a totally positive
	generator. Thus the set of non-zero   ideals may be identified with $ \SO^+ / \RU^2 $ (by definition, $\mathrm{U}^2 =  \{ \epsilon^2 : \epsilon \in \mathrm{U}  \}$).  
	
	Let $S_k(\mathrm{SL}_2 (\SO))$ denote the space of Hilbert modular cusp forms  of parallel even weight $(k,k)$ with respect to the Hilbert modular group $\mathrm{SL}_2 (\SO)$. Let $ H_k $ be a Hecke orthonormal basis of $S_k(\mathrm{SL}_2 (\SO))$. For each cusp form $f \in H_k$, let $\omega_f$ be its harmonic weight and %$\psi_f(\mu)$ be its normalized Fourier coefficients and 
	$\lambda_f(\mu)$ be its Hecke eigenvalues ($\mu   \in \SO^+ / \RU^2$). 
	
	The standard $L$-function associated to $f$ is defined by
	\begin{equation*}
		L(s,f) = \sum_{\mu   \in \SO^+ / \RU^2 }  \lambda_f(\mu ) N(\mu )^{-s}, \qquad \text{($\mathrm{Re} (s)>1$)}.
	\end{equation*}
	In 1977, Asai \cite{Asai-1977} introduced a kind of sub-series of $L(s,f) $, in which the summation is restricted to rational integers only. More explicitly,  Asai's $L$-function $L (s, \mathrm{As}(f))$ is defined by
	\begin{equation*}
		L(s,\As(f) ) = \zeta(2s) \sum_{n \in \BZ_{+}}  \lambda_f(n) n^{-s}, \qquad \text{($\mathrm{Re} (s)>1$)},
	\end{equation*}
and he proved many fundamental results of  $ L (s, \mathrm{As}(f)) $, such as  analytic continuation, functional equation, Euler product, and splitting formula.

  Asai's  $L$-function and its generalizations have been studied extensively in the framework of representation theory, %the literature, 
usually in the adelic language.  
 Ramakrishnan \cite{Ramakrishnan2002} and Krishnamurthy \cite{Krishnamurthy2003} proved that  $L(s,\As(f) )$ is in fact the $L$-function associated to an automorphic form on $\GL_4 (\mathbf{A}_{\bfQ})$, namely, the Asai lift $\mathrm{As}(f)$ of $f$.  Moreover, Prasad and Ramakrishnan \cite[Appendix A]{Krishnamurthy-2012} established the cuspidal criterion for $\As(f) $ as summarized in  \cite[Theorem 1.1]{Luo-Asai}.

 However, it is only recent  that Luo \cite{Luo-Asai} initiated the study on the analytic aspects  of the family of central $L$-values  $ L(1/2, \As(f) )  $. He proved the sharp mean-Lindel\"of bound for the second moment:
 \begin{equation}\label{1eq: Luo's corollary}
 	\sumx_{f\in H_k} \left| L(1/2, \As(f)) \right|^2 \Lt_\upvarepsilon k^{2+\upvarepsilon } ,
 \end{equation}
for any $\vepsilon > 0$, where the $\star$  means that the summation is restricted to cuspidal Asai lifts $\As(f)$. 
Luo's result is based on his large sieve inequality
\begin{equation}\label{1eq: Luo's result}
	\sum_{f\in H_k } \omega_f   \bigg| \sum_{N (\mu) \leqslant  X} b_{\mu} \lambda_f (\mu )\bigg|^2 \Lt_\upvarepsilon   \big( k^2 + X  \big)  (kX   )^\upvarepsilon \sum_{ N (\mu) \leqslant  X } | b_{\mu} |^2,
\end{equation}
where %$\omega_f$ are the harmonic weights,  
$\boldsymbol{b} : \SO^{+} / \RU^2 \ra \BC$ is an arbitrary complex  sequence %, the $h$ means that the $f$-sum is harmonic weighted by $\omega_f$,  
and it is understood that the $\mu$-sum is over $ \SO^{+} / \RU^2 $.  This should really be  viewed as the large sieve associated to the family of standard $L$-functions $ L (s, f) $. For any given $\vnu \in \SO^+$, if \eqref{1eq: Luo's result} is specialized to those $\mu = \vnu n$ ($n \in \BZ_{+}$) (in other words, if $\boldsymbol{b}$ is supported on $ \vnu \BZ_+$),  then 
\begin{equation}\label{1eq: Luo LS, Asai}
	\sum_{f\in H_k } \omega_f \bigg| \sum_{n \leqslant Y } a_{n}  \lambda_f ( \vnu n )\bigg|^2 \Lt_\upvarepsilon   \big( k^2 +  Y^2 N (\vnu)   \big)  (  k  Y  N (\vnu) )^\upvarepsilon \sum_{ n \leqslant  Y } |a_{n} |^2,
\end{equation}
where $\boldsymbol{a} : \BZ_+ \ra \BC$ is an arbitrary complex  sequence. For $\vnu = 1$, we consider \eqref{1eq: Luo LS, Asai} as a benchmark large sieve inequality associated to $ L (s, \As(f)) $. Thus \eqref{1eq: Luo's corollary} follows from \eqref{1eq: Luo LS, Asai} via a standard argument by approximate functional equation with $Y = k^{1+\vepsilon}$ (see \cite{Luo-Asai}).

\begin{remark}\label{rem: QxQ}
	Note that if $\bfF / \BQ$ were replaced by the split quadratic algebra $\BQ \times \BQ$, then the analogue of   $L (1/2,  \mathrm{As}(f) )$ would be the Rankin--Selberg   $L (s, h \times h)$, for $ h \in S_{k}  (\mathrm{SL}_2 (\BZ))$. It is also known that $  L (1/2,  \mathrm{As}(f) ) $ may be expressed in terms of $ L (1/2, \mathrm{Sym}^2 h  ) $ if $f$ is the base change of $h$ {\rm(}see \cite[\S 5]{Asai-1977}{\rm)}.

	From the analytic perspective, both $L (1/2,  \mathrm{As}(f) )$ and $L (1/2, h \times h)$ or $L (1/2, \allowbreak \mathrm{Sym}^2 h  )$ are examples of $L$-values of `conductor drop',  but the latter has attracted more attention due to its connection with quantum unique ergodicity {\rm(}QUE{\rm)}. 
	
	It should be stressed here the difference in dimensions  {\rm(}see \cite{Shimizu,Siegel-Vol}\footnote{There is somehow a factor $2$ missing in the volume formula for $\mathrm{Vol} (\SL_2 (\SO) \backslash \mathbf{H}^2)$ in \cite[\S 1.11]{Garrett-Hilbert} compared to that of Siegel \cite[(19)]{Siegel-Vol}.}{\rm)}:  \begin{align*}
		   \dim S_k (\mathrm{SL}_2 (\BZ)) &  \sim \frac {\mathrm{Vol} (\SL_2 (\BZ) \backslash \mathbf{H} )} {4 \pi }   k = \frac k  {12} ,  \\
		     \dim  S_k (\mathrm{SL}_2 (\SO))  & \sim \frac {\mathrm{Vol} (\SL_2 (\SO) \backslash \mathbf{H}^2)} {(4\pi)^2}   k^2 = \frac {\zeta_{\bfF} (2) d_{\bfF}^{3/2}}  {8 \pi^4} k^2 ,
	\end{align*}
where $d_{\bfF}$ is the discriminant of $\bfF$. 
	
The reader is referred to \cite{Young-Sym2-LS} for the large sieve inequality for $ \mathrm{Sym}^2 h $  and to \cite{Khan-Young-Sym2} for the best known result for the mean square of $ L (1/2, \mathrm{Sym}^2 h)$. 
\end{remark}

\subsection{Main Results} 

Let $\phi $ be a fixed Hecke--Maass cusp form  in $S_{\frac 1 4 + \vvkappa^2 }  (\mathrm{SL}_2 (\BZ))$. In this paper, we consider the second moment of convoluted $L$-functions $ L (s, \mathrm{As}(f) \times \phi)$ at the center $s = 1/2$.   

\begin{theorem}\label{thm: 2nd moment}
	We have
	\begin{equation}\label{1eq: 2nd moment bound}
		\sumx_{f\in H_k}  | L  (1/2, \As(f)\times  \phi  )  |^2 \Lt_{\vepsilon, \phi}  k^{7/2+\upvarepsilon} ,
	\end{equation}
 where the $\star$ indicates the restriction to cuspidal Asai lifts $\As(f)$.
\end{theorem}

\begin{remark}
	Note that  in order to apply the converse theorem of Cogdell and Piatetski-Shapiro \cite{C-PS-Converse-GL(4)},  it has been verified in \cite{Ramakrishnan2002}  the `niceness'  of  $ L (s, \mathrm{As}(f) \times \phi)$, in particular, its analytic continuation and functional equation. 
\end{remark}

\begin{remark}\label{rem: Asai and RS}
	By Remark \ref{rem: QxQ}, one may regard $ L (s, \mathrm{As}(f) \times \phi) $ as a sibling of the triple product $L$-function  $ L (s, h \times h \times \phi) $ or the convoluted $L$-function $ L (s, \mathrm{Sym}^2 h \times \phi) $, but for the latter, to the authors' knowledge,  there is currently no result  in the literature like Theorem \ref{thm: 2nd moment}, but some related ones may be found for example in \cite{BKY-Mass,Khan-Young-Sym2}. For comparison, applying the holomorphic analogue of Young's large sieve for $ \mathrm{Sym}^2 h $  (\cite{Young-Sym2-LS}) %, with $h $ traversing a Hecke basis of $  S_{k}  (\mathrm{SL}_2 (\BZ))$, 
	only yields the trivial bound $O_{\vepsilon, \phi} (k^{3+\vepsilon})$ for the second moment of $ L (1/2, \mathrm{Sym}^2 h \times \phi) $.  
\end{remark}

\begin{remark}
	 It is tempting to use the Petersson--Vorono\"i approach to improve the bound \eqref{1eq: 2nd moment bound}, but  in the literature  there is   no Vorono\"i  summation formula for $\SL_2 (\BZ) $ that involves additive characters over $\mathbf{F}$. 
\end{remark}

The benchmark large sieve in \eqref{1eq: Luo LS, Asai}, applied with  $Y $ up to $ k^{2+\vepsilon}/N(\vnu)$,  would only yield  $ O_{\vepsilon, \phi} (k^{4+\vepsilon})$, which is trivial by the convexity bound $L  (1/2, \As(f)\times  \phi  ) = O_{\vepsilon, \phi}  (k^{1+\vepsilon})$.   Note that \eqref{1eq: 2nd moment bound} is just the  Burgess subconvexity bound on average.

The proof of \eqref{1eq: 2nd moment bound} will rely on the large sieve inequality in the next theorem. 

 \begin{theorem}\label{thm: large sieve}
 	Let $\boldsymbol{a} : \BZ_+ \ra \BC$ {\rm(}$\boldsymbol{a} = \{ a_{n} \} ${\rm)} be an arbitrary complex sequence. Let  $\vnu \in \SO^+$.  Then  for $Y > k \geqslant 2$ {\rm(}$k$ even{\rm)}, we have
 	\begin{equation}\label{1eq: main LS}
 		\sum_{f\in H_k} \omega_f   \bigg| \sum_{n\leqslant Y} a_n \lambda_f(\vnu n ) \bigg|^2 \Lt_\upvarepsilon \max \left\{ \sqrt{kY}, \frac {Y} {\sqrt{k}} \right\}  ( Y N(\vnu) )^{1+\upvarepsilon} \sum_{n\leqslant Y} |a_n|^2 ,
 	\end{equation}
 where the implied constant depends on $\vepsilon $ only. 
 \end{theorem}

Our \eqref{1eq: main LS} is stronger than Luo's \eqref{1eq: Luo LS, Asai} for all $ Y > k  $, and its proof will be entirely different from his, in particular, it requires an average over all the weights $K < k \leqslant K + H$.  

 \begin{theorem}\label{thm: large sieve, average}
 Let $\boldsymbol{a} : \BZ_+ \ra \BC $ and $\vnu \in \SO^+ $ be as above in Theorem \ref{thm: large sieve}.  Then for $ 1 \leqslant H \leqslant K  < Y 
 $ we have
 \begin{equation}\label{eq: large sieve, averaged}
 	\begin{split}
 		\sum_{ K<k\leqslant K+H}  & \sum_{f\in H_k} \!  \omega_f  \bigg| \sum_{n\leqslant Y} a_n \lambda_f(\vnu n ) \bigg|^2 \! \!  \Lt_{\vepsilon} \! \!  \bigg(  H K  +  \frac {Y } {H }  +  \frac{Y}{\sqrt{  K}}  \bigg)   (Y N(\vnu))^{1+\upvarepsilon} \! \! \sum_{n\leqslant Y} |a_n|^2,
 	\end{split}
 \end{equation}
where it is understood that the $k$-sum is over $2 \BZ_+$.  
\end{theorem}

Now \eqref{1eq: main LS} follows from \eqref{eq: large sieve, averaged} if we choose $H = \min \big\{ \sqrt{K}, \sqrt{Y/K} \big\}$ and ignore all but one term in the $k$-sum by non-negativity. However, if we choose $H = 1$ so that there is only one $k$ left, then we arrive instead at \eqref{1eq: Luo LS, Asai}, so the average over $k$ is quite essential. 

\begin{remark}
	In view of Proposition {\rm\ref{prop: reduction}}, we may improve the bound   {\rm\eqref{eq: large sieve, averaged}} in the $N(\vnu)$-aspect, but this is not beneficial in our application. 
\end{remark}
 
\begin{remark}
	We may also prove the short-interval variant of {\rm\eqref{1eq: 2nd moment bound}} by a direct application of Theorem \ref{thm: large sieve, average}{\rm:}
		\begin{equation}\label{1eq: 2nd moment bound, averaged}
	\sum_{ K<k\leqslant K+\sqrt{K}}  \ 	\sumx_{f\in H_k}  | L  (1/2, \As(f)\times  \phi  )  |^2 \Lt_{\vepsilon, \phi}  K^{7/2+\upvarepsilon} . 
	\end{equation}
\end{remark}

\subsection{Remarks on the Proof of Theorem \ref{thm: large sieve, average}} 

The difficulty of the problem lies in the analysis of  Bessel products $J_{k-1} (x) J_{k-1} (y)$ for 
\begin{align*}  
	x = \frac{4\pi \vnu \sqrt{  m n} }{ c  }, \qquad y = \frac{4\pi \vnu' \sqrt{  m n} }{ c'  },  
\end{align*}
arsing from the application of the Petersson formula for $\SL_2 (\SO)$ (here $c'$ is the conjugate of $c$ in $\mathbf{F}$). Olver's uniform asymptotic formula is a useful tool, but its behavior varies across different ranges, not to speak of the  simultaneous variation of  $x$ and $y$.  

Our analysis will rely on an integral representation of the smoothed variant of 
\begin{align*}
	\sum_{ K<k\leqslant K+H} J_{k-1} (x) J_{k-1} (y), 
\end{align*}
deduced from the Macdonald integral formula. By careful stationary phase analysis, it turns out that if $x$ and $y$ are not too close to each other, the sum above is negligibly small unless $x , y > K^{1+\vepsilon}$. Thus $x$ and $y$ are both beyond the transition range of Olver's asymptotic formula! As for the case that $x$ and $y$ are very close, the integral representation will still be useful in our analysis.

\subsection*{Notation}   
By $F \Lt G$ or $F = O (G)$ we mean that $|F| \leqslant c G$  for some constant $c  > 0$, and by $F \asymp G$ we mean that $F \Lt G$ and $G \Lt F$. We write $F \Lt_{\valpha, \beta, ...} G $ or $  F = O_{\valpha, \beta, ...} (G) $ if the implied constant $c$ depends on $\valpha$, $\beta$, ....  

The notation $x \sim X$ stands for  $ X <  x \leqslant 2 X $.

By `negligibly small' we mean $ O_A ( K^{-A} )$ or $O_A (k^{-A})$ ($  O_A (\vkappa^{-A})$ as we shall set $\vkappa = k-1$) for arbitrarily large but fixed $A > 0$. 

Throughout the paper,  $\upvarepsilon $ is arbitrarily small and its value may differ from one occurrence to another.

	\section{Preliminaries}
\subsection{Basic Definitions}\label{sec: basic defn}
Fix a real quadratic field $\bfF = \bfQ (\sqrt{D})  $ ($D$ square-free). Assume that the narrow class number $h_{\bfF }^{ +} = 1$.  Let $\SO $, $\SO^+$, and $\RU$ denote the ring of integers, the set of totally positive integers, and  the group of units, respectively. Consider $ \SO $ as a lattice in the plane $\BR^2$ and $\SO^+$ embedded in the first quadrant $\BR_+^2$.  Note that $ \RU = \{ \pm 1 \} \times \langle \epsilon_0 \rangle  $ while $ \RU^2 = \langle \epsilon_0^2 \rangle  $ so that $[\RU : \RU^2] = 4$, where as usual $\epsilon_0 $ is the fundamental unit. 

Let $\xi'$ denote the conjugate of $\xi$ in $\bfF$. Let $N(\xi) = \xi \xi'$ and $  \Tr(\xi) = \xi + \xi'$ be the norm and the trace for $\bfF \subset \BR^2 $($= \bfF_{\infty}$), respectively. 

Let $d_{\bfF}$ be the discriminant of $\bfF$. Then $\sqrt{d_{\bfF}}$   generates the different ideal of $\bfF$. As the narrow class number $h_{\bfF }^{ +} = 1$, according to \cite[\S 26.8]{Hasse-NT},  $d_{\mathbf{F}}$ has only one prime divisor (hence $D$ must be prime), so either $d_{\bfF} = 8$ ($D = 2$) or $d_{\bfF} = D$ is a prime $ \equiv 1 (\mathrm{mod}\, 4)$. 

Write $e (x) = \exp (2\pi i x) $. For $x \in \BR^2$,  it will be convenient to introduce $e_{\bfF}  [x] = e (\Tr ( x / \sqrt{d_{\bfF}})) $, viewed as an additive character on $ \BR^2 $($= \bfF_{\infty}$).

For $\mu \in \SO \smallsetminus \{0\}$, let $r(\mu) \in \BZ_{+}$ denote the largest rational integer dividing $\mu$. 

As usual, for $\vnu, \mu \in  \SO \smallsetminus \{0\}$,  write $ \vnu \sim \mu $ if $ (\vnu) = (\mu)$. For $ \vnu, \mu \in \SO^{+} $, clearly $ \vnu \sim \mu $ if and only if $\vnu / \mu \in \mathrm{U}^2$. 

For  $\vnu, \, \mu \in\SO$ and $c \in \SO \smallsetminus \{0\}$ define the Kloosterman sum
\begin{align}\label{2eq: Kloosterman}
	S_{\bfF}  (\vnu, \mu ; c ) = \sumx_{\valpha (\mathrm{mod}\, c) } e_{\bfF}  \bigg[   \frac {\vnu \valpha + \mu \widebar{\valpha} } {c } \bigg] , 
\end{align}
where the $\star$ indicates the condition $(\valpha,c)= (1)$ and $\widebar{\valpha}$ is given by $\valpha \widebar{\valpha}  \equiv 1 (\mathrm{mod} \, c)$. We have the Weil bound
\begin{align}
	\label{2eq: Weil}
	S_{\bfF}  (\vnu, \mu ; c )   \Lt \tau (c) \sqrt{N((\vnu, \mu, c))} \sqrt{|N (c)|}. 
\end{align}

Let $\pi $ and $p$ denote primes in $ \SO^{+} /\RU^2 $ and $\BZ_+$ respectively. 

	\subsection{Petersson Formula for Hilbert Cusp Forms} 
	
	Next, we briefly review the Petersson trace formula for Hilbert cusp forms on $\mathbf{H}^2$ with respect to $ \SL_2 (\SO)$; as usual $\mathbf{H}$ is the hyperbolic upper half-plane. For more details, the reader is referred to \cite{Luo-2003-Hilbert}, \cite{Luo-Asai}. 

As before, for $k \in 2 \BZ_+$ let $H_{k}$ be an orthonormal basis for $S_k (\SL_2 (\SO))$---the space of Hilbert modular cusp forms of weight $(k,k)$. Every $f (z) \in H_{k}$ has  Fourier expansion of the form:
\begin{equation*}
	f(z) = \sum_{\mu \in \SO^+ } a_f(\mu) e \big( \Tr  ( \mu z / {\textstyle \sqrt{d_{\bfF}} }) \big), \qquad \text{($z \in \mathbf{H}^2$)}.  
\end{equation*}
%here $\Tr (z) = z_1 + z_2 $ for $z = (z_1, z_2)$,  $z_1, z_2 \in \mathbf{H}$. 
It is known that the Fourier coefficients $ a_f (\mu)$ are invariant under $ \mathrm{U}^2$, that is $ a_f (\epsilon^2 \mu) = a_f (\mu) $ for any $\epsilon \in \mathrm{U}$.  Assume that every $f (z)$ is a Hecke eigenform of all the Hecke operators. Let $\lambda_f (\mu)$ ($\mu \in \SO^{+}/ \RU^2 $) be its Hecke eigenvalues so that $ a_f (\mu) = a_f (1) \allowbreak \lambda_f (\mu) N(\mu)^{(k-1)/2}$. It is known that $ \lambda_f (\mu) $ are all real-valued. 

The harmonic weight is defined by
\begin{align*}
	\omega_f = \frac{\Gamma (k)^2 d_{\bfF}^{k+1} }{(4\pi)^{2(k-1)} } |a_f(1)|^2. 
\end{align*} 
Note that $ \omega_f $ only plays a very minor role because we have a direct extension of the bounds of Iwaniec, Hoffstein, and Lockhart \cite{Iwaniec-L(1),HL-L(1)}
\begin{align}
k^{-\vepsilon} \Lt	\omega_f \Lt k^{\vepsilon} . 
\end{align}

 \delete{ We have  the Hecke relation
	\begin{equation}\label{2eq: Hecke relation}
		\lambda_f(\vnu) \lambda_f(\mu) =  \sum_{\delta \in \SO^+/ \RU^2   : \, \delta | (\vnu, \mu) }   \lambda_f \Big( \frac{\vnu \mu}{\delta^2} \Big),
	\end{equation}
	and the Ramanujan bound (\cite{Taylor-Hilbert})
	\begin{align}
		\lambda_f(\mu) \Lt_{\vepsilon} N(\mu)^{\vepsilon} . 
\end{align} }

Let $\vnu, \mu \in \SO^{+}$. According to \cite{Luo-2003-Hilbert,Luo-Asai}, the Petersson trace formula reads
	\begin{equation}\label{2eq: Petersson}
	\begin{split}
		\frac 1 { (k-1)^2 } \sum_{f \in H_k} \omega_f &   \lambda_f   (\vnu) {\lambda_f (\mu)}   =  d_{\bfF}^{3/2} \cdot \delta (\vnu  \sim \mu )    \\
		&  + 4 \pi^2 d_{\bfF}  \sum_{\epsilon\, \in \RU/ \{\pm 1\}}  \sum_{c\, \in \SO^+ / \RU^2 } \frac{ S_{\bfF} (\vnu \epsilon, \mu \epsilon ;c) }{N(c)}   N J_{k-1} \bigg( \frac{4\pi \sqrt{\vnu \mu \epsilon^2 }  }{c} \bigg) , 
	\end{split}
\end{equation}
where  $\delta (\vnu \sim \mu)$ is the Kronecker $\delta$-symbol that detects $\vnu \sim \mu$ (or $\vnu/\mu \in \RU^2$),  and  
\begin{align}
	N J_{k-1} \bigg( \frac{4\pi \sqrt{\vnu \mu \epsilon^2 }  }{c} \bigg) = J_{k-1} \bigg( \frac{4\pi \sqrt{\vnu \mu \epsilon^2 }  }{c} \bigg) J_{k-1} \bigg( \frac{4\pi \sqrt{\vnu' \mu' \epsilon'^2 }  }{c' } \bigg). 
\end{align}

\begin{remark}
	Note that the sum over $ \epsilon \in \mathrm{U} $ in \cite[(9)]{Luo-2003-Hilbert} or \cite[(2)]{Luo-Asai} should instead be taken over  $ \epsilon \in \mathrm{U} /\{ \pm 1 \} $ since Luo inadvertently miscounted the center $\{ \pm \boldsymbol{1}_2 \}$ of $ \SL_2 (\SO) $. Actually, the formula \cite[(9)]{Luo-2003-Hilbert} for $\mathbf{F} = \mathbf{Q}$ differs slightly from the classical Petersson formula as in \cite[(14.15)]{IK}. 
\end{remark}

\begin{remark}
	For simplicity, we have restricted % $\epsilon $ on $\RU / \{ \pm \}$ and 
	the $c$-sum  on $ \SO^+ / \RU^2 \approx (\SO \smallsetminus \{0\} ) / \RU $ so as to avoid the use of $| \ |$ and written $ S_{\bfF} ( \vnu \epsilon, \mu \epsilon ;c )$ instead of  $ S_{\bfF} ( \vnu , \mu \epsilon^2 ;c )$ so as to facilitate our later analysis.   Subsequently, the representatives $c$ in $ \SO^+ / \RU^2 $ will  be chosen so that 
	\begin{align}\label{2eq: c, c'}
		c,\,  c'  \asymp \sqrt{N(c)} . 
	\end{align} 
\end{remark}

\subsection*{Abbreviation} For brevity, we shall henceforth  write  
\begin{align}
	\vkappa = k -1 . 
\end{align}

\subsection{Basic Properties of Bessel Functions} Next, we recollect some results of the Bessel function $J_{\vkappa} (x)$ of  positive argument $x \in \BR_+$ and integral order $\vkappa \in \mathbf{Z}_{+}$. Note that $ J_{\vkappa} (-x) = (-1)^{\vkappa} J_{\vkappa} (x) $, so sometimes $ x \in  \BR_+ $ may be   relaxed into  $ x \in  \BR $.

Firstly, in our preliminary analysis (see \S \ref{sec: Petersson}), we need the uniform crude bounds 
\begin{align}\label{2eq: prelim bounds}
	|J_{\vkappa}(x)| \leqslant 1, \qquad |J_{\vkappa} (x)| \Lt \Big( \frac{e x}{2 \vkappa}\Big)^{\vkappa}, 
\end{align}
which, by trivial estimation (and for the latter by the Stirling  formula), may be deduced from the Bessel and Poisson  integral representations (see \cite[2.2 (2), 3.3 (1)]{Watson}):  \begin{align*}%\label{2eq: Bessel's integral}
	J_{\vkappa}(x) = \frac{1}{\pi} \int_{0}^{\pi} \cos(\vkappa \theta - x \sin\theta) \nd \theta.
\end{align*}
\begin{align*}
	J_{\vkappa}(x) = \frac{  (x/2)^{\vkappa} }{ \sqrt{\pi}\, \Gamma(\vkappa+1/2)  } \int_{0}^{\pi } \cos(x \cos\theta) \sin^{2\vkappa}\theta \, \nd \theta.
\end{align*}
It is clear from \eqref{2eq: prelim bounds}  that $J_{\vkappa}(x)$ is negligibly small unless $x > \vkappa/2$.

Our analysis will rely crucially on the  Macdonald integral representation\footnote{It is interesting to note that this formula  is also used (in the reverse direction) by Iwaniec and Li  \cite{Iwaniec-Li-Ortho} in their study of the orthogonality of Hecke eigenvalues of holomorphic cusp forms.} (\cite[13.7(1)]{Watson}):
\begin{equation}\label{2eq: Macdonald}
		J_{\vkappa} (x) J_{\vkappa} (y) = \frac{1}{2\pi i^{\vkappa+1}} \int_{ {-\infty}}^{\infty}  { \exp \bigg( \frac {i} {2} \bigg( {r}  + \frac{x^2 + y^2}{  r}\bigg) \bigg) J_{\vkappa} \Big( \frac{xy}{ r}   \Big) } \frac{\nd r}{r},
\end{equation}
for $x, y \in \BR  $. %\red{Check $I_{\vkappa} (\pm i x)$ and $J_{\vkappa} (x)$, $x \in \BR_+$! }

For $x \in \BR_+$, let us also record the Mehler--Sonine  integral representation (see \cite[6.21(12)]{Watson}):
\begin{equation}\label{2eq: integral of J_0}
	J_0(x) = \frac{2}{\pi} \int_0^{\infty} \sin ( x \cosh t ) \nd t.
\end{equation} 
Moreover, we may write (see  \cite[\S 7.21]{Watson}) 
\begin{align}\label{3eq: B+-, x>1}
	J_0 (x) =  \frac 1 {\sqrt{x}} \big( \exp (ix) W_{+} (x) +  \exp (- ix) W_{-} (x) \big) ,  
\end{align}
so that 
\begin{align}\label{3eq: bounds W, x>1}
	x^j W^{   (j)}_{\pm} (x) \Lt_{ j}  1  , \qquad x \Gt 1. 
\end{align}

\subsection{Sums of Bessel Functions} 

For $h \in C_c^\infty(\BR_+)$, we have the following Bessel summation formula as in \cite[\S 5.5]{Iwaniec-Topics} (see also \cite[Lemma 4.1]{Luo-Sarnak-2003}):
\begin{equation}\label{eq: Luo_4.1}
	\sum_{2 \, \nmid  \, \vkappa  } 2\pi i^{\vkappa+1} h(\vkappa) J_{\vkappa}(x) =  -   \int_{-\infty}^{\infty} \widehat{h}(t) \sin(x \cos t ) \nd t, 
\end{equation}
where % $\widehat{h}$ is the Fourier transform of $h$, namely,
\begin{align}\label{2eq: Fourier}
	  \widehat{h}(t) = \int_{-\infty}^{\infty} h(y) \exp(i ty) \nd y . 
\end{align}

\subsection{Olver's Uniform Asymptotic Formula for Bessel Functions}  

As illustrated by Olver's uniform asymptotic formula, the behavior of $ J_{\vkappa} (\vkappa x) $ is very delicate if $x $ is close to $ 1$. However, it will be shown dramatically in Lemma \ref{lem: B(x, y), ranges} that one only needs to focus   on the case $ x > \vkappa^{\vepsilon} $, which is beyond the transition range so that Olver's   asymptotic is simple.   %As aforementioned, we only need to focus on the case  $x > 1/2$. 
It will be more convenient to apply Olver's results in a less explicit form  as in the following lemma. 

\begin{lem}\label{lem: Olver}
	   For $x > 1$ define \begin{align}\label{2eq: gamma}
	 	 \gamma (x) = \sqrt{x^2-1} - \mathrm{arcsec}\, x  . 
	 \end{align} 
Then for $ x > 2$, we have
\begin{align}\label{Beq: Olver}
	J_{\vkappa} (\vkappa x) = \frac {1} { \sqrt{\vkappa x}  } \bigg\{   { \exp (   i \vkappa \gamma (x) ) }  V_{+} (x) + { \exp ( - i \vkappa \gamma (x) ) }  V_{-} (x)  + O_A \bigg(\frac{1}{\vkappa^A  } \bigg)\bigg\} ,
\end{align}
for any $A>0$, in which  \begin{align}\label{2eq: V+-}
	 x^j V_{\pm}^{(j)} (x) \Lt_{j} 1  {\rm;}
\end{align} in particular, we have bound
\begin{align}\label{2eq: bound, Olver}
	 J_{\vkappa} (\vkappa x) \Lt \frac 1 {\sqrt{\vkappa x} }. 
\end{align}
\end{lem}

The reader is referred to Appendix \ref{app: Olver} for the Olver uniform asymptotic formula for $ J_{\vkappa} (\vkappa x) $ (in the case $x > 1$) and the deduction of Lemma \ref{lem: Olver} (in the case $x > 2$).

\subsection{Variant of Gallagher's Hybrid Large Sieve}

The hybrid large sieve inequality for Dirichlet characters of Gallagher \cite{Gallagher-LS} is a widely used tool, and it has been extended to number fields by Duke \cite{Duke-LargeSieve}. Their method works for additive characters as well, and, on the quadratic field $\bfF$, the hybrid large sieve reads:  
\begin{equation*}
	\sum_{\valpha (\mod\, c)} \int_{-T}^{T}  \bigg| \sum_{N(\mu) \leqslant X} b_\mu  N(\mu)^{it} e_{\bfF} \Big[ \frac{\mu \valpha}{c} \Big]  \bigg|^2 \nd t \Lt  ( T N(c) + X  ) \sum_{N(\mu) \leqslant X} |b_\mu|^2 ,
\end{equation*}
where $\boldsymbol{b} : \SO^{+} / \RU^2 \ra \BC$ is an arbitrary complex  sequence.  Thus, if it is specialized to those $\mu = \vnu n$ ($n \in \BZ_{+}, \, \vnu \in \SO^+$), then
\begin{equation*}
	\sum_{\valpha (\mod\, c)} \int_{-T}^{T}  \bigg| \sum_{n \leqslant Y} a_n n^{it} e_{\bfF}  \Big[\frac{\vnu n \valpha}{c} \Big] \bigg|^2 \nd t \Lt \big( T N(c) + Y^2 N(\vnu) \big) \sum_{n \leqslant Y} |a_n|^2 ,
\end{equation*}
for arbitrary $\boldsymbol{a} : \BZ_+ \ra \BC$, but this bound is too weak for our purpose in the $Y$-aspect. 

The next lemma will be proven by the method of Gallagher in Appendix \ref{appendix: large sieve}. Note that $Y^2$ above is now improved into $Y r (c)$.  

	\begin{lem}\label{C lem: Gallagher thm 2}
	Let $\boldsymbol{a} : \BZ_+ \ra \BC$ be an arbitrary complex sequence. Let $c \in \SO^+ $. As in \S \ref{sec: basic defn}, define $r(c)$ to be the largest rational divisor of $c$. Assume $T, Y \geqslant 1$.  Then
	\begin{equation}\label{Ceq: large sieve}
		\sum_{\valpha (\mod\, c)} \int_{-T}^{T}  \bigg| \sum_{n \leqslant Y} a_n n^{it} e_{\bfF} \Big[ \frac{n \valpha}{c} \Big] \bigg|^2 \nd t \Lt  ( T N(c) + Y r(c)   ) \sum_{n \leqslant Y} |a_n|^2 .
	\end{equation}
\end{lem}

	\begin{coro}\label{C cor: Gallagher thm 2}
	Let notation be as in Lemma \ref{C lem: Gallagher thm 2}.   Then for $\vnu \in \SO^{+}$ we have 
	\begin{equation}\label{eq: cor: large sieve}
		\sum_{\valpha (\mod\, c)} \int_{-T}^{T}  \bigg| \sum_{n \leqslant Y} a_n n^{it} e_{\bfF} \Big[ \frac{\vnu n \valpha}{c} \Big] \bigg|^2 \nd t \Lt \big( T N(c) + Y r(c_{\vnu}) N((\vnu,c)) \big) \sum_{n \leqslant Y} |a_n|^2 ,
	\end{equation}
where $c_{\vnu}$ is defined so that  $(c) =  (\vnu,c)  (c_{\vnu}) $. 
\end{coro}

	\begin{proof}
	%Choose $c^\star \in \SO^+$ so that $(c) =  (\vnu,c)  (c^\star) $. Now 
	The left side of \eqref{eq: cor: large sieve} is equal to 
	$$ \sum_{\valpha (\mod\, c)} \int_{-T}^{T}  \bigg| \sum_{n \leqslant Y} a_n n^{it} e_{\bfF} \bigg[ \frac{ n \valpha}{c_{\vnu}} \bigg] \bigg|^2 \! \nd t = N ((\vnu, c)) \! \sum_{\beta (\mod\, c_{\vnu})} \int_{-T}^{T}  \bigg| \sum_{n \leqslant Y} a_n n^{it} e_{\bfF}  \bigg[\frac{ n \beta }{c_{\vnu}} \bigg] \bigg|^2 \! \nd t. $$
By Lemma \ref{C lem: Gallagher thm 2}, this is bounded by
	$$ N((\vnu,c)) \big( T N(c_{\vnu}) + Y r(c_{\vnu})  \big) \sum_{n \leqslant Y} |a_n|^2 , %\leqslant \big( T N(c) + Y r(c) N((\vnu,c)) \big) \sum_{n \leqslant Y} |a_n|^2 , 
	$$
	and hence the desired bound as $N((\vnu,c)) \cdot  N(c_{\vnu}) = N (c)$. % and $ r (c^{\star}) \leqslant r (c)  $. 
\end{proof}

%The hybrid large sieve inequality of Gallagher \cite{Gallagher-LS} is a widely used tool, and it has been extended to number fields by Duke \cite{Duke-LargeSieve}. On the quadratic field $\bfF$, the hybrid large sieve reads: \red{Where is it in Duke??????}

\section{Initial Reductions}\label{sec: initial reductions}

For technical simplifications, we shall henceforth replace the range $n\leqslant Y$ in Theorems \ref{thm: large sieve} and \ref{thm: large sieve, average} by the dyadic $Y < n \leqslant 2Y$ ($n \sim Y$). It is clear that the results so modified are of the same strength. Denote 
\begin{equation*}
	\VA =  \bigg( \sum_{n\sim Y} |a_n|^2 \bigg)^{1/2}.
\end{equation*}
 
 Let us also make the mild assumption $Y^{\vepsilon},  N (\vnu)^{\vepsilon} < K $ as otherwise we may deduce   \eqref{1eq: Luo LS, Asai} from \eqref{eq: large sieve, averaged} (by the $\vepsilon$-convention).

\subsection{Set-up} \label{sec: setup}
 Let $K^\upvarepsilon \leqslant  H \leqslant K^{1-\vepsilon}$.  For $\varww \in C_c^\infty(\BR)$ define the weight function
\begin{equation}\label{eq: def of test function h}
	h(x) = \varww \bigg( \frac{x - K}{H}\bigg),
\end{equation}
and consider the smoothly weighted sum 
\begin{equation}\label{3eq: S(a)}
S_{\vnu} (\boldsymbol{a}) =	  \sum_{2|k} \frac {h(\vkappa)}  {\vkappa^2 } \sum_{f\in H_k} \omega_f  \bigg| \sum_{n \sim Y} a_n \lambda_f(\vnu n ) \bigg|^2,  \qquad \text{($\vkappa = k-1$)}.
\end{equation}
%For example, we may choose $ \varww $ to be a bump function which dominates the characteristic function of a fixed interval.  

\subsection{Application of the Petersson Formula} \label{sec: Petersson}

By opening the square in \eqref{3eq: S(a)} and applying the Petersson trace formula in \eqref{2eq: Petersson}, we infer that 
\begin{align}
	S_{\vnu} (\boldsymbol{a}) =   d_{\bfF}^{3/2}    D  (\boldsymbol{a}) + 4\pi^2 d_{\bfF}  \sum_{\epsilon \in \RU /\{\pm 1\}}   P_{\epsilon \vnu} (\boldsymbol{a}), 
\end{align}
where 
\begin{align}
	D (\boldsymbol{a}) = \sum_{2 \, \nmid  \, \vkappa  }   h(\vkappa) \sum_{n \sim Y} |a_n|^2, 
\end{align}
\begin{align}\label{3eq: P(a)}
	P_{\vnu} (\boldsymbol{a}) = \sum_{2 \, \nmid  \, \vkappa } h(\vkappa)  \mathop{\sum \sum}_{m,n \sim Y} a_m \overline{a}_n   \sum_{c\, \in \SO^+ / \RU^2}   \frac{ S_{\bfF} (\vnu m ,  \vnu n ;c)  }{N(c)}  NJ_{\vkappa} \bigg( \frac{4\pi \sqrt{  \vnu^2  m n} }{ c  } \bigg). 
\end{align}
It is clear that the diagonal sum $D (\boldsymbol{a}) = O (H \VA^2 )$.  Moreover, the $\epsilon$-sum over units in the off-diagonal summation does not play an essential role, so one may absorb $\epsilon$ into $\vnu$ as suggested by the notation. This is because the Bessel product  
$	N J_{\vkappa} \big(  {4\pi \sqrt{   \epsilon^2 \vnu^2  m n} } / { c  } \big) $  in the sum $P_{\epsilon \vnu} (\boldsymbol{a})$ is negligibly small unless 
\begin{align*}
	  |\epsilon \vnu| , \, |\epsilon' \vnu'|   > \frac {\vkappa}  {8\pi e Y}, 
\end{align*}
by the bounds in \eqref{2eq: prelim bounds}, but there are at most $O (\log K)$ many such units $\epsilon$ in this range  as $ \RU / \{\pm 1\}$ is the cyclic group $\langle \epsilon_0 \rangle$.  %Note that the desired estimate involves only $N (\vnu)$.  
For simplicity, let us still restrict ourselves to  the case $ \vnu \in \SO^{+} $ in the study of $P_{\vnu} (\boldsymbol{a}) $ so as to avoid writing $| \ |$ everywhere (the reader may follow our subsequent analysis in the case $\vnu = 1$  as well).

	\subsection{Truncation and Partition}  \label{sec: partition}
	Recall that $J_{\vkappa}(x)$ is negligibly small for $x \leqslant \vkappa/2$ (see \eqref{2eq: prelim bounds}), so we may restrict the $c$-sum in \eqref{3eq: P(a)} to the range
	\begin{equation}\label{eq: range of c}
		N(c) \Lt \frac {N(\vnu) Y^2} {K^2} , 
	\end{equation}
	at the cost of a negligible error. 
Keep in mind that (see \eqref{2eq: c, c'})  $$c, c' \asymp \sqrt{N(c)}. $$ For dyadic $C \Lt N(\vnu) Y^2/K^2$, define
\begin{equation}\label{3eq: D(C)}
	\BSD(C) = \big\{ c\in \SO^+  : N(c) \sim C, \, c,c'\asymp \sqrt{C} \big\}.
\end{equation} 

Further, in order to facilitate our later analysis, we need to partition the $c$-sum according to the value of \begin{align}\label{3eq: theta}
	\theta = \frac {\vnu' c  }  {\vnu c '  }. 
\end{align} More explicitly, we introduce
	\begin{align} 
	\label{3eq: D infty}	& \BSD_\vnu^{0}   = \big\{ c\in \SO^+ : \vnu' c / \vnu  c '  < 1/3  \big\},  \\
		\label{3eq: D 1}	& \BSD_\vnu^1 = \big\{ c\in \SO^+: | 1-  \vnu' c / \vnu  c ' | \Lt H^\upvarepsilon/H  \big\},  \\
	\label{3eq: D Delta}	&\BSD_\vnu (\varDelta)   = \big\{ c\in \SO^+:  1 - \vnu' c / \vnu  c ' \sim \varDelta /H  \big\} ,
\end{align}  
for dyadic $H^\upvarepsilon \leqslant \varDelta \leqslant H $.  	Accordingly, we partition $P_\vnu(\boldsymbol{a})$ as in \eqref{3eq: P(a)} into the sum of $P_\vnu^{0} (\boldsymbol{a})$,   $P_\vnu^1(\boldsymbol{a})$, and  those $P_\vnu(\boldsymbol{a};\varDelta)$. Set 
\begin{align}\label{3eq: D1(C), D(C, Delta)}
   \BSD_\vnu^1 (C) = \BSD(C) \cap \BSD_\vnu^1,  \qquad 	 \BSD_\vnu (C,\varDelta) = \BSD(C) \cap \BSD_\vnu (\varDelta), 
\end{align} 
and let  $P_\vnu^1(\boldsymbol{a};C) $ and  $P_\vnu(\boldsymbol{a};C,\varDelta)$ denote their contributions to $P_\vnu^1(\boldsymbol{a})$ and $P_\vnu(\boldsymbol{a};\varDelta)$   respectively. Since $\theta \theta ' = 1$, we have focused mainly in the case $ \theta < 1 $ by symmetry.

\subsection{Counting Lattice Points} 

Later, it will be required to know the number of points in $\BSD_\vnu^1 (C)$ and $\BSD_\vnu (C,\varDelta)$. To this end, let us prove here a simple lemma. 

\begin{lem}\label{lem:lattice_count}
	Let  $0 < \delta \Lt 1 \Lt C$.  Then for any  $\xi \in \BR^{\times}$ 
	$$  \text{\small \bf \#}  \left\{ c \in \SO : c, c' \asymp \sqrt{C}, \, \left| 1 - \xi \frac{c }{c' }   \right| \leqslant \delta \right\}  \Lt C\delta + \sqrt{C},$$
	where the implied constant depends only on the field $\bfF$.
\end{lem}

%\subsection{Reduction} By the discussion above, the problem is reduced to proving the following bound for $ P_{\vnu} (\boldsymbol{a})   $. 

\begin{figure}[ht]
	\centering
	\begin{tikzpicture}[scale=0.4] 
		\def\vx{1.2} \def\vy{0.2}   
		\def\ux{0.3} \def\uy{1.2}   
		
		\def\ang{6.2} 
		
		\coordinate (P1) at ({2.7*\vx + 1.3*\ux}, {2.7*\vy + 2.7*\uy});
		\coordinate (P2) at ({7.7*\vx + 1.3*\ux}, {7.7*\vy + 2.7*\uy});
		\coordinate (P3) at ({7.7*\vx + 7.7*\ux}, {7.7*\vy + 7.7*\uy});
		\coordinate (P4) at ({2.7*\vx + 7.7*\ux}, {2.7*\vy + 7.7*\uy});
		
		\begin{scope} 
			\clip (P1) -- (P2) -- (P3) -- (P4) -- cycle;
			
			\fill[lightgray] (0,0) -- (15, {15*tan(45-\ang)}) -- (15, {15*tan(45+\ang)}) -- cycle;
		\end{scope}
		
		\draw[thin] (-0.08,-0.08) -- (9, {9*tan(45+\ang)});
		\draw[thin] (-0.08,-0.08) -- (12, {12*tan(45-\ang)});
		
		\draw[thin] (P1) -- (P2) -- (P3) -- (P4) -- cycle;
		
		\foreach \i in {-1,...,8}
		\foreach \j in {-1,...,8}
		{ 
			\fill[black!80] ({\i*\vx + \j*\ux}, {\i*\vy + \j*\uy}) circle (1.4pt);
		}
		
		\node[below] at (0,0) {\footnotesize $O$};
		
		\draw [draw = black, fill = white] (0, 0) circle (0.1);
		
	\end{tikzpicture}
	\caption{Area $\EuScript{R}_{\xi} (C, \delta)$.}
	\label{fig:region_R_clean}
\end{figure}

\begin{proof}

	Consider $ \SO \hookrightarrow  \BR^2 $ as a lattice  so that the problem is reduced to counting lattice points in the (enlarged) area $\EuScript{R}_{\xi} (C, \delta)$ depicted in Figure \ref{fig:region_R_clean}. More explicitly, $\EuScript{R}_{\xi} (C, \delta)$  is the intersection of a parallelogram of side  lengths $\asymp \sqrt{C}$ and an origin-centered sector of angle $\asymp \delta$.  Note that its area 
	\begin{align*}
		\text{\it Area}\, (\EuScript{R}_{\xi} (C, \delta)) \asymp (\sqrt{C})^2 \cdot \delta = C \delta,  
	\end{align*}
while its perimeter 
\begin{align*}
	\text{\it Per}\, (\EuScript{R}_{\xi} (C, \delta)) \Lt \sqrt{C} .   
\end{align*}
Then this lemma follows from a fundamental result in the geometry of numbers known as the Lipschitz principle \cite[Theorem]{Davenport-Lipschitz}: For a lattice $\Lambda \subset \BR^2$ and a region $\EuScript{R} \subset \BR^2$ (under certain conditions (I, II) as in \cite{Davenport-Lipschitz}), 
\begin{align*}
	 \text{\small \bf \#} (\Lambda \cap \EuScript{R}) - \frac { 	\text{\it Area}\, (\EuScript{R} )} { \text{\it Area}\, (\Lambda \backslash \BR^2 ) } \Lt 1 + \text{\it Per}\, (\EuScript{R} ). 
\end{align*}
\end{proof} 

By applying Lemma \ref{lem:lattice_count} with $\xi = \vnu' /\vnu $ and  $\delta \asymp H^{\upvarepsilon}/H$ or $\delta \asymp \varDelta/H$, we obtain bounds for the cardinalities of   $\BSD_\vnu^1 (C)$ and  $\BSD_\vnu (C,\varDelta)$ as follows. 

\begin{coro}\label{cor:lattice_count}
	 We have
	 \begin{align*}
	 	\text{\small \bf \#}  \BSD_\vnu^1 (C) \Lt   C H^{\upvarepsilon}/H + \sqrt{C}, \qquad	\text{\small \bf \#}  \BSD_\vnu (C,\varDelta) \Lt  C \varDelta/H + \sqrt{C} . 
	 \end{align*}
\end{coro}

\subsection{Reductions} 

The rest of this paper will be mainly devoted to the proof of the following bounds for $P_\vnu^{0} (\boldsymbol{a})$, $P_\vnu^1(\boldsymbol{a})$, and $P_\vnu(\boldsymbol{a};\varDelta)$.

\begin{proposition}\label{prop: reduction}
	Let $Y, K, H,  \varDelta$ be large parameters such that   %\begin{align}\label{3eq: conditions} 
	$	  K^\upvarepsilon \leqslant H \leqslant K^{1-\vepsilon}$, and  $ H^\upvarepsilon \leqslant \varDelta \leqslant H$. 
%	\end{align} 
Then for $\vnu \in \SO^+$, we have
	\begin{equation}\label{eq: estimate for P_infty}
		P_\vnu^{0} (\boldsymbol{a}) \Lt_{\vepsilon} \frac{Y^2  }{H K^2 } \sqrt{N(\vnu)} Y^\upvarepsilon \VA^2 ,
	\end{equation} 
\begin{equation}\label{eq: estimate for P_v^1}
	P_\vnu^1(\boldsymbol{a}) \Lt_{\vepsilon}  \bigg( \frac {Y^2} {H K^2} N(\vnu) + \frac {Y} {K} \sqrt{N(\vnu)}  + \frac {H Y} {K}   \bigg) Y^\upvarepsilon \VA^2 , 
\end{equation}
	\begin{equation}\label{eq: estimate for P_Delta}
		P_\vnu(\boldsymbol{a};\varDelta) \Lt_{\vepsilon}  \bigg(  \bigg(\frac{H}{\sqrt{  K}} +  1  \bigg) \bigg(\frac {Y^2} {HK^2} N(\vnu)+ \frac {Y} {K} \sqrt{N(\vnu)} \bigg) + \frac {HY} {K} \bigg)  Y^\upvarepsilon \VA^2 .
	\end{equation}
\end{proposition}

Theorem \ref{thm: large sieve, average} follows if we choose $\varww \in C_c^\infty(\BR)$ as in \eqref{eq: def of test function h} to be a suitable bump function, multiply $K^2$ to   $S_{\vnu} (\boldsymbol{a}) $ as in \eqref{3eq: S(a)}, and apply the bounds in Proposition \ref{prop: reduction}.

	\section{Integral Representation for Sums of Bessel Products} \label{sec: integral repn}
	
	For  $h \in C_c^\infty(\BR_+)$ and $x, y \in \BR $, consider the sum
	\begin{equation}\label{eq: def of B}
		B(x,y) = \sum_{2 \, \nmid  \, \vkappa }  h(\vkappa) J_{\vkappa}(x) J_{\vkappa}(y),
	\end{equation} 
	which arose partially in the definition of 
	$P_{\vnu} (\boldsymbol{a}) $ as in \eqref{3eq: P(a)}. Keep in mind that in our later sections, 
	\begin{align}\label{4eq: x, y}
		x = \frac{4\pi \vnu \sqrt{  m n} }{ c  }, \qquad y = \frac{4\pi \vnu' \sqrt{  m n} }{ c'  }. 
	\end{align}
	
	\begin{lem}\label{lem: representation of B}
	Define for $t, x, y \in \BR$ 
	\begin{equation}\label{eq: def of omega(t)}
		\omega (t;x,y) = \sqrt{x^2 + y^2 - 2xy\cos t }.
	\end{equation} 
Then 
		\begin{equation}\label{4eq: integral expression of B}
			B(x,y) =   \frac{ 1 }{4 \pi } \int_{- \infty}^{ \infty } \widehat{h}(t)  \big( J_0( \omega ( t;x, y) ) - J_0( \omega ( t;x,-y) )   \big)  \nd t,
		\end{equation}
		where $  \widehat{h}(t) $ is the Fourier transform of $h $ as in {\rm\eqref{2eq: Fourier}}. 
		
	\end{lem}
	\begin{proof}
By the Macdonald integral formula in \eqref{2eq: Macdonald},  
\begin{equation*}
	B(x,y) =  \frac{1}{2\pi } \int_{ {-\infty}}^{\infty}  { \exp \bigg( \frac {i} {2} \bigg( {r}  + \frac{x^2 + y^2}{  r}\bigg) \bigg) \Bigg\{\sum_{2 \, \nmid  \, \vkappa } i^{\vkappa+1} h(\vkappa) J_{\vkappa} \Big( \frac{xy}{ r}   \Big) } \Bigg \} \frac{\nd r}{r} .
\end{equation*}
Next, the inner $\vkappa$-sum may be evaluated by \eqref{eq: Luo_4.1}, yielding
\begin{equation*}
	B(x,y) = - \frac{1}{4\pi^2 } \int_{ {-\infty}}^{\infty} \widehat{h}(t)  \int_{ {-\infty}}^{\infty}   \exp \bigg( \frac {i} {2} \bigg( {r}  + \frac{x^2 + y^2}{  r}\bigg) \bigg) \sin \lp \frac {xy \cos t } {r} \rp  \frac{\nd r \nd t}{r} . 
\end{equation*}
Finally, as a consequence of  \eqref{2eq: integral of J_0}, for $\omega \in \BR_+$ we have
\begin{equation*}
	J_0(\omega) = \frac{1}{2 \pi i} \int_{-\infty}^{\infty} \exp \bigg( \frac{i}{2} \bigg( r +  \frac{\omega^2}{r}   \bigg) \bigg) \frac{\nd r}{r} ,
\end{equation*}  from which \eqref{4eq: integral expression of B} follows  by a simple calculation. 
	\end{proof}

\subsection{Bounds for $  \omega (t; x, y) $} 

\begin{lem}\label{lem: omega}
	 Let $ |t | < \delta  $ be small. Then  
	  \begin{align}\label{4eq: bounds omega}
	 	  \omega (t; x, y) \Gt |x  - y |, \qquad  \omega' (t; x, y) \Lt \frac { \delta |xy |  } {|x - y |} ,  
	 \end{align} 
 \begin{align}
 		\label{4eq: bounds omega, 2} \omega^{(j)} (t; x, y) \Lt_j & \frac { {|xy |}  } {|x - y |} \bigg( 1 +  \frac {   \sqrt{|xy |}  } {|x - y |}  \bigg)^{j-2}, 
 \end{align}for any $j \geqslant 1$.

\end{lem}

\begin{remark}
	In practice, the case when $x, y$ have different signs is much simpler. Note that {\rm\eqref{4eq: bounds omega}} and {\rm\eqref{4eq: bounds omega, 2}} now read  
	 \begin{align}
		\omega (t; x, y) \Gt |x - y |, \qquad  \omega' (t; x, y) \Lt \frac { \delta |xy |  } {|x - y |} , \qquad \omega^{(j)} (t; x, y) \Lt_j \frac {   |xy |  } {|x - y |}  , 
	\end{align}
	due to $ \sqrt{|xy|} \Lt |x-y| $. % {\rm(}the last bound may  be slightly improved into $ |xy|/|x-y|${\rm)}. 
\end{remark}

\begin{proof}
Consider the case when $ x , y$ have the same sign. 	We have
	\begin{align*}
		\omega (t; x, y) = {  \sqrt{ (x-y)^2 + 4 xy \sin^2 (t/2)}}, \qquad \omega' (t, x, y) =   \frac {x y \sin t} {  \omega (t; x, y)   }, 
	\end{align*}
hence 
\begin{align*}
	\omega' (t; x, y) \Lt \min \bigg\{ \sqrt{|xy|},  \frac { \delta |xy |  } {|x - y |}   \bigg\}. 
\end{align*} 
Moreover, we may prove by induction (or by  the Fa\`a di Bruno formula \cite{Faa-di-Bruno}) that $ \omega^{(j)} (t; x, y) $ is a linear combination of 
\begin{align*}
	\frac { (xy)^{m + n} \sin^{m} t \cos^{n} t } { \omega (t; x, y)^{2m+2n -1 } } , \qquad 0 < m + 2 n \leqslant j . 
\end{align*}
It follows that  
\begin{align*}
	  \omega^{(j)} (t; x, y) & \Lt \frac { |xy| |\sin t |   }  {  |x-y| } +  \mathop{\sum \sum}_{ 1 < m + 2 n \leqslant j } \frac { |xy|^{m+n} |\sin^{m} t |  }  { {|xy|^{m/ 2}}  |\sin^{m} (t/2) | \cdot |x-y|^{m + 2n-1}  } \\
	  & \Lt \frac { |xy| }  {  |x-y| } +  \mathop{\sum \sum}_{ 1 < m + 2 n \leqslant j } \frac { |xy|^{m/2+n} }  {   |x-y|^{m + 2n-1}  }   \\ 
	  & \Lt  \frac {   |xy |  } {|x-y|}   \bigg( 1 + \frac {\sqrt{|xy|}} {|x-y|} \bigg)^{j- 2} ,
\end{align*} 
for any $j \geqslant 2$. 
%The bounds above for $ \omega (t; x, y) $ follow from some direct calculations. 
\end{proof}

\subsection{Further Analysis} 
Next, we need to carefully analyze  the integral representation for $B (x, y)$ in \eqref{4eq: integral expression of B}; in particular, we would like to determine the ranges of $x, y$ on which $B (x, y)$ is not negligibly small. 

Let $K^\upvarepsilon \leqslant H \leqslant K^{1-\vepsilon}$. For our choice of $ h (\vkappa) $   in \S \ref{sec: setup}, 
\begin{equation} \label{eq: h hat h}
	h(\vkappa) = \varww \bigg( \frac{\vkappa - K}{H}\bigg), \qquad \widehat{h}(t) = H \cdot \exp(i Kt) \widehat{\varww}(Ht). 
\end{equation}
Thus the $t$-integral may be effectively truncated at $ |t| = H^{\vepsilon} / H $ at the cost of a negligible error.  Further,  as in  \eqref{3eq: B+-, x>1},  there is also oscillation in  $ J_0( \omega ( t;x, \pm y) ) $.  It follows that $B (x, y)$ splits into four similar integrals, one of which is of the form:
\begin{align}\label{4eq: B++(x,y)}
	B^{+}_{+} (x, y) = H \int_{- H^{\vepsilon} / H}^{H^{\vepsilon} / H} g_{H}^{+} (  t; x, y) \exp \big(i K t + i \omega (t; x, y) \big) \nd t , 
\end{align}
where 
\begin{align}
	g_{H}^{+} (t; x, y) = \frac {\widehat{\varww}(H t)\cdot W_{+} (\omega (t ; x, y)) } {4\pi \sqrt{ \omega (t ; x, y) }} . 
\end{align}
By the bounds in \eqref{3eq: bounds W, x>1}, \eqref{4eq: bounds omega}, and  \eqref{4eq: bounds omega, 2}, it is routine to verify using the product rule and the Fa\`a di Bruno formula \cite{Faa-di-Bruno} that
\begin{align}\label{4eq: g^(j)_+}
	\frac {\partial^j g_{H}^{+  }  (t; x, y)}  {\partial t^j } \Lt_{j}  \frac{1}{  \sqrt{|x - y|} }  \bigg( H + \frac {\sqrt{|xy|}} {|x-y|} \bigg)^{j} , 
\end{align}
provided   $ |x - y | \Gt 1 $.  Now assume that $x , y > K / 2$ say, as $ B (x, y) $ is negligibly small if otherwise (by \eqref{2eq: prelim bounds}).  

\begin{lem}\label{lem: B(x, y), ranges}
	Let $x > y > K/2$. 
	
	{\rm(1)} In the case $x / y > 3$,   $ B (x, y) $ is negligibly small unless 
	\begin{equation*}
		x,y \Gt   {H K } / {H^\upvarepsilon} .
	\end{equation*}

 {\rm (2)}
In the case  $ x/y - 1  \sim \varDelta/H $ {\rm(}for $ H^{\vepsilon'} \! \leqslant \varDelta \leqslant H ${\rm)}, $B(x,y)$ is negligibly small unless
\begin{equation*}
	x,y \Gt   {\varDelta K} / {H^{\upvarepsilon  }}  .
\end{equation*}
\end{lem}

\begin{proof}
	 Let us consider only $ B^{+}_{+} (x, y) $ in \eqref{4eq: B++(x,y)}, as the other three integrals may be treated similarly. Now up to the factor $1/2\pi$ the phase function of the integral reads
	 $$ f (t) = K t + \omega(t;x,y). $$
	 By \eqref{4eq: bounds omega} in Lemma \ref{lem: omega}, with $\delta = H^\upvarepsilon / H$,  we have
	\begin{align*}
		  |\omega'(t;x,y)| \Lt \frac{\delta xy }{ x-y } \Lt \left\{  \begin{aligned}
		  	 &   {H^{\vepsilon} y} / {H}, & &  \text{ if } x/y > 3,   \\ 
		  	 &  {H^{\vepsilon} y} / {\varDelta}, & & \text{ if }  x/y - 1  \sim \varDelta/H.   
		  \end{aligned} \right. 
	\end{align*}
Thus  \begin{align*}
	 f' (t) = K +\omega' (t; x, y) \Gt K, 
\end{align*}  respectively for $  y \Lt HK / H^{\vepsilon} $ or  $ y \Lt \varDelta K / H^{\vepsilon} $ in the first   or  the second case. 
For  $y$ in such ranges, in view of the bounds in   \eqref{4eq: bounds omega, 2} and \eqref{4eq: g^(j)_+},  we conclude that $ B^{+}_{+} (x, y)  $ is negligibly small by applying Lemma \ref{lem: staionary phase, dim 1} with $R  = K$, $P = 1/H$, $Q = 1$ or $\varDelta / H $, and $Z = HK $ or $\varDelta^2 K / H  $.  Note here that 
\begin{align*}
	 \frac {\sqrt{xy }} { x-y } %= \frac  {\sqrt{x/y}} {x/y-1} 
	\Lt  \left\{  \begin{aligned}
	 	&  1, & &  \text{ if } x/y > 3,   \\ 
	 	&   H / {\varDelta}, & & \text{ if }  x/y - 1  \sim \varDelta/H.  \end{aligned} \right. 
\end{align*}
\end{proof}

Lemma \ref{lem: B(x, y), ranges} manifests that $x, y > K/ 2$ is strengthened at least into $ x,  y > K^{1+\vepsilon} $ (in (2) choose  $\vepsilon' = 2 \vepsilon$, say) in case that $x$ and $ y$ are not too close. This will enable us to apply Olver's asymptotic formula for $ J_{\vkappa} (x) $ and $J_{\vkappa} (y) $ simultaneously beyond the transition range at $\vkappa$ as in Lemma \ref{lem: Olver}.

%\section{Bounds for Bessel Products} 

%Next, let us consider the problem of bounding the Bessel product $ J_{\vkappa} (x) J_{\vkappa} (y)  $. 

\subsection{Bounds for $B (x, y)$} For $|x-y| \Gt 1$, if trivial estimation were applied to \eqref{4eq: B++(x,y)}, we would only get 
\begin{align*}
	B^{+}_{+} (x, y) \Lt   H \int_{- H^{\vepsilon} / H}^{H^{\vepsilon} / H} \frac {\nd t} {  \sqrt{ |x-y| + \sqrt{xy} |t| }} \Lt H^{\vepsilon} \min \bigg\{ \frac {1} {\sqrt{|x-y|}} , \frac {\sqrt{H}} {\sqrt[4]{xy}}
	\bigg\}  .
\end{align*} 
For  $ x, y > K^{1+\vepsilon} $, however, this is worse than the following simple bound obtained trivially  from \eqref{2eq: bound, Olver} and the definition in \eqref{eq: def of B}:
\begin{align}\label{4eq: bound for B(x, y)}
	 B (x, y) \Lt \frac {H} {\sqrt{xy}}. 
\end{align}

	\section{\texorpdfstring{Estimates for $P_\vnu^{0} (\boldsymbol{a})$}{Estimates for P\unichar{"2070}\unichar{"03BD}(a)}  }\label{sec: Estimate for P^infty}
	
By the definitions in \eqref{3eq: P(a)}, \eqref{3eq: D infty}, and \eqref{eq: def of B}, we have
	\begin{equation*}
	P_\vnu^{0} (\boldsymbol{a}) =  \mathop{\sum \sum}_{m,n \sim Y}  a_m \overline{a}_n  \sum_{ c\, \in \BSD_\vnu^{0} / \RU^2  } \frac{  S_{\bfF} (\vnu m, \vnu n;c)  }{N(c)}  B   \bigg( \frac{4\pi \vnu \sqrt{m n} }{c } ,\frac{4\pi \vnu' \sqrt{m n} }{c' } \bigg) . 
\end{equation*}
By Lemma \ref{lem: B(x, y), ranges} (1), the $c$-sum may be further truncated and restricted to the range 
\begin{align*}
	N(c) \Lt \frac {N(\vnu) Y^{2+\upvarepsilon} }{H^2 K^2 } ,
\end{align*}
so, compared to  \eqref{eq: range of c}, we have saved  $   H^2$ for the length of summation! By the Weil bound in \eqref{2eq: Weil} and the bound in \eqref{4eq: bound for B(x, y)},  trivial estimation yields
\begin{align*}
	 P_\vnu^{0} (\boldsymbol{a}) & \Lt Y^{\vepsilon}  \mathop{\sum \sum}_{m,n \sim Y}  | a_m \overline{a}_n |  \sum_{ (c) : N (c / \vnu) \Lt   {  Y^{2+\upvarepsilon} }  / {H^2 K^2 }  } \frac { \sqrt{ N ((\vnu m, \vnu n, c)) } } {\sqrt{N (c)}} \cdot   \frac{H \sqrt{N(c)} }{ Y \sqrt{  N(\vnu) } } \\
	 & \Lt Y^{\vepsilon} \frac{H }{ Y   } \frac {\sqrt{N(\vnu)} Y^2} {H^2 K^2}   \mathop{\sum \sum}_{m,n \sim Y} \tau (   (\vnu m, \vnu n)) \cdot  | a_m \overline{a}_n |   \\
	 & \Lt   \frac { Y^2 } {H K^2} \sqrt{N(\vnu)} Y^{\vepsilon}  \VA^2, 
\end{align*}
as in \eqref{eq: estimate for P_infty}.

\section{Mellin Technique} 

For the estimations of $P_{\vnu}^{1}(\boldsymbol{a})$ and $P_{\vnu}(\boldsymbol{a};\varDelta)$ in the next two sections, the variant of Gallagher's hybrid large sieve will be applied, so the Mellin inversion will be required  to separate the entangled variables $m$ and $n$ in $\sqrt{mn}$. To this end, let us establish here two lemmas by analyzing certain Mellin integrals.  

\begin{lem}\label{lem: Mellin of exp}
	Let $S, X>0$. Let $\varww (x ) \in C^{\infty} [X/2,  4 X]$ % {\rm(}not necessarily compactly supported{\rm)} 
	satisfy $ \varww^{(j)}(x) \Lt_j S / X^j $. Then for $x \sim X$, we have
	\begin{equation}
		\exp(i x) \varww (x)   = \int_{-\infty}^{\infty} \xi(r) x^{ i r} \nd r,
	\end{equation}
	where $\xi(r)$ has bound $\xi(r) \Lt S (1 + |r| + X )^{-A}$ for any $A > 0$ unless $ r  \asymp X$, in which case $\xi(r) \Lt S /\sqrt {1 + X }$.
\end{lem}

	\begin{proof}
As $x$ is restricted to   $[X , 2 X]$, we may assume that $ \varww (x) \in C_c^{\infty} [X/2,  4 X] $ is compactly supported by multiplying a suitable $ \varvv (x) \in C_c^{\infty} [X/2, 4 X]  $ with $ \varvv (x)  \equiv 1 $ on $[X, 2 X]$.  	
Now by Mellin inversion,  
	\begin{equation*}
		\xi(r) = \frac{1}{2\pi} \int_{0}^{\infty}   \exp(i x - i r \log x) \varww (x)   \frac{\nd x}{x}.
	\end{equation*}
Note that trivially $ \xi (r) \Lt S $.	The first bound for $\xi (r)$  is an easy consequence of Lemma \ref{lem: staionary phase, dim 1}. % with phase function $f (x) = x + r \log x$. 
	The bound $ \xi (r) \Lt S /\sqrt{1 + X} $ follows from the second derivative test in Lemma \ref{lem: 2nd derivative}. 
\end{proof}

	\begin{lem}\label{lem: Mellin of w_kappa}
	Define 
	\begin{equation}\label{6eq: def of gamma_kappa}
		\gamma_{\vkappa}(x)=\sqrt{x^{2}-\vkappa^{2}}-\vkappa \operatorname{arcsec}(x/\vkappa).
	\end{equation}
	Let $X > \vkappa^{1+\upvarepsilon}$ and $1 - \theta   \sim \delta$ with $1/\vkappa \Lt \delta \Lt 1/\vkappa^\upvarepsilon$. Let $S$ and $\varww (x ) $ be as in Lemma \ref{lem: Mellin of exp}. Then for $x \sim X$, we have
	\begin{align}\label{6eq: lem: applying Mellin}
	\exp(  i\gamma_{\vkappa}( x) \pm i\gamma_{\vkappa}(\theta x)) \varww (x )   = \int_{-\infty}^{\infty} \xi^{\pm} (r)  x^{i r} \nd r, 
	\end{align}
where $\xi^+(r)$ or $\xi^-(r)$ respectively is negligibly small unless $r  \asymp X$ or  $ \delta X$, in which case 
\begin{align}
	\xi^{+} (r) \Lt \frac S {\sqrt{X}}, \qquad \xi^{-} (r) \Lt S. 
\end{align}
\end{lem}

\begin{proof}
	As in the proof of Lemma  \ref{lem: Mellin of exp}, we may assume with no loss of generality that $ \varww (x) \in C_c^{\infty} [X/2,  4 X] $. By Mellin inversion, 
	 we have
	\begin{equation*}
		\xi^{\pm}(r) = \frac{1}{2\pi} \int_0^\infty \exp(  i\gamma_{\vkappa}( x) \pm i\gamma_{\vkappa}(\theta x) - ir\log x ) \varww (x) \frac{\nd x}{x} .
	\end{equation*}
	The phase function here (up to the factor $1/2\pi$) equals
	\begin{equation*}
		f_{\pm}(x;r) =  \gamma_\vkappa( x) \pm \gamma_\vkappa( \theta  x  ) - r\log x.
	\end{equation*}
	By direct calculations,  
	\begin{equation*}
		f_{\pm}'(x;r) = \frac{\sqrt{ x^2 - \vkappa^2} \pm \sqrt{ \theta^2 x^2 - \vkappa^2} - r}{x}, 
	\end{equation*}
	\begin{equation*}
		f_{\pm}''(x;r) = \frac{1}{x^2} \bigg( \frac{\vkappa^2}{\sqrt{x^2 - \vkappa^2} } \pm \frac{\vkappa^2}{\sqrt{\theta^2  x^2 - \vkappa^2} } + r \bigg) ,
	\end{equation*}
	and 
	\begin{equation*}
	x^j	f_{+}^{(j)}(x;r) \Lt_j \frac{\vkappa^2}{x } +  {|r|}  , \qquad x^j f_{-}^{(j)}(x;r) \Lt_j \frac{ \delta \vkappa^2}{x } +  {|r|} ,
	\end{equation*}
for every $j \geqslant 2$. Note that
\begin{align*}
	f_{\pm}'(x;r) = (1 \pm \theta ) (1+ O (1/\vkappa^{\vepsilon})) - r / x   . 
\end{align*} Then we can apply Lemma \ref{lem: staionary phase, dim 1} with $R= 1 + |r|/X$,
$ P = Q = X$,  and $ Z = \vkappa^2/X + |r| $ 
to deduce that $\xi^+(r)$ is negligibly small unless $ r \asymp X$. The bound $\xi^+(r) \Lt S/ \sqrt{X}$ follows readily  from Lemma \ref{lem: 2nd derivative}. Similarly, we apply Lemma \ref{lem: staionary phase, dim 1} with $R= \delta + |r|/X$, 
$ P = Q = X$, and $ Z = \delta \vkappa^2/X + |r| $
to deduce that $\xi^-(r)$ is negligibly small unless $r \asymp \delta X$. The bound $\xi^-(r) \Lt S$ however is trivial. 
\end{proof}

	\section{\texorpdfstring{Estimates for $P_\vnu^1(\boldsymbol{a})$}{Estimates for P\unichar{"00B9}\unichar{"03BD}(a) } }\label{sec: Estimate for P^1}
	
After the partition in \S \ref{sec: partition}, for $ C \Lt  N(\vnu) Y^2/K^2$ let us consider the sum 
	\begin{equation}\label{7eq: P_1}
		P_\vnu^1 (\boldsymbol{a}; C) =  \mathop{\sum \sum}_{m,n \sim Y}  a_m \overline{a}_n  \sum_{ c\, \in \BSD_\vnu^{1} (C) / \RU^2  } \frac{  S_{\bfF} (\vnu m, \vnu n;c)  }{N(c)}  B   \bigg( \frac{4\pi \vnu \sqrt{m n} }{c } ,\frac{4\pi \vnu' \sqrt{m n} }{c' } \bigg). 
	\end{equation}
See \eqref{3eq: P(a)}, \eqref{3eq: D(C)}, \eqref{3eq: D 1},  \eqref{3eq: D1(C), D(C, Delta)}, and \eqref{eq: def of B}.  

%For $C \Lt H^2 / K^{\vepsilon}$, similar to the case of $P_\vnu^{0} (\boldsymbol{a})$, we shall estimate the sum trivially. For $ H^2 / K^{\vepsilon} \Lt C \Lt N(\vnu) Y^2/K^2 $ (it was assumed in \eqref{3eq: conditions}  that $ H K \leqslant Y $), we shall apply the integral representation for $B (x, y)$ and then the variant of Gallagher's hybrid large sieve under the integral. 

\delete{%\subsection{Trivial Estimation for $C  $ Small}  Note that $ x , y > K^{1+\vepsilon} $ is satisfied for $ C \Lt H^2 / K^{\vepsilon} $ (see \eqref{4eq: x, y}). By the Weil bound in \eqref{2eq: Weil} and the bound in \eqref{4eq: bound for B(x, y)}, along with Corollary \ref{cor:lattice_count}, trivial estimation yields
\begin{align*}
	\begin{aligned}
		P_{\vnu}^1 (\boldsymbol{a}; C) & \Lt Y^{\vepsilon}  \mathop{\sum \sum}_{m,n \sim Y}  | a_m \overline{a}_n |  \sum_{ c\, \in \BSD_\vnu^{1} (C) } \frac { \sqrt{ N ((\vnu m, \vnu n, c)) } } {\sqrt{N (c)}} \cdot   \frac{H \sqrt{N(c)} }{ Y \sqrt{  N(\vnu) } } \\
	& \Lt Y^{\vepsilon} \frac{H  }{ Y \sqrt{  N(\vnu) }  }   \bigg( \frac{C}{H} + \sqrt{C} \bigg)   \mathop{\sum \sum}_{m,n \sim Y} \tau (   (\vnu m, \vnu n)) \cdot  | a_m \overline{a}_n |  , 
	\end{aligned}
\end{align*}
and it follows from $ C \Lt H^2 / K^{\vepsilon} $ that
\begin{align}%\label{7eq: estimate for P_1 for C small}
	 P_\vnu^{0} (\boldsymbol{a}; C) \Lt \frac { H^2 } {\sqrt{N(\vnu)} }   Y^{\vepsilon}  \VA^2. 
\end{align}}
	
	\subsection{Application of the Integral Representation}

		%Next we consider the case $ H^2 / K^{\vepsilon} \Lt C \Lt N(\vnu) Y^2/K^2 $. 
		By inserting \eqref{4eq: integral expression of B} into \eqref{7eq: P_1}, we see that $ P_\vnu^1(\boldsymbol{a}, C)$ splits into two similar integrals, one of which reads 
	\begin{equation}\label{7eq: integral representation of P_1}
		P_{\vnu}^{+}(\boldsymbol{a};C) = \frac{1}{4\pi } \int_{- \infty}^{\infty}  \widehat{h}(t)  S_{\vnu} \big( t  ; \boldsymbol{a}; C \big)   \nd t,
	\end{equation}
	where 
	\begin{align}\label{7eq: S (C)}
	S_{\vnu} (t  ; \boldsymbol{a}; C) =	\sum_{ c\, \in \BSD_\vnu^{1}(C) / \RU^2  }  \frac{ S_\vnu  ( 4\pi \omega  (t, \vnu / c , \vnu' /c'   )   ; c , \boldsymbol{a}  )  }{N(c) }   , 
	\end{align}
and 
	\begin{equation}
		S_\vnu (\beta; c , \boldsymbol{a}) = \mathop{\sum \sum}_{m,n \sim Y}  a_m \overline{a}_n  S_{\bfF} (\vnu m, \vnu n ;c)  J_0(\beta \sqrt{mn}).
	\end{equation}
	
	\subsection{Application of the Hybrid Large Sieve} 
Let us now estimate the bilinear form $S_\vnu (\beta; c , \boldsymbol{a})$ with the aid of the hybrid large sieve in Corollary \ref{C cor: Gallagher thm 2}. Note that a similar one over $\BZ$ was studied in \cite[Proposition 3]{D-I-Kuz}. 
	
	\begin{lem}\label{lem: estimate of S_v beta}
		Let $c\in \BSD_\vnu^1 $ {\rm (}see \eqref{3eq: D 1}{\rm )} and $\vnu \in \SO^+$. Then for any $\beta \in \BR_+ $, we have
		\begin{equation}\label{eq: lem: estimate of S_v beta}
			S_\vnu (\beta; c , \boldsymbol{a}) \Lt_{\vepsilon}  \big( N(c) + \min\{1/\beta , Y  \} \cdot r(c_{\vnu} ) N((\vnu,c))  \big) Y^\upvarepsilon \VA^2 . 
		\end{equation}
%	with the implied constant independent on $\beta$. 
	\end{lem}

\begin{proof}
	Let us first consider the case $\beta \geqslant 1/Y$. By \eqref{3eq: B+-, x>1}, we split $S_\vnu (\beta; c , \boldsymbol{a}) $ into the sum of $ S_\vnu^{\pm} (\beta; c , \boldsymbol{a})  $ given by 
	\begin{equation*}
		S_\vnu^{\pm} (\beta; c , \boldsymbol{a}) =   \mathop{\sum \sum}_{m,n \sim Y}  \frac{a_m \overline{a}_n}{\sqrt{\beta} \sqrt[4]{mn} } S_{\bfF} (\vnu m, \vnu n ;c)  \exp(\pm i \beta \sqrt{mn}) W_{\pm}(\beta \sqrt{mn}).
	\end{equation*}
Now Lemma \ref{lem: Mellin of exp} yields, up to a negligibly small error, the expression  
	\begin{equation*}
		S_\vnu^{\pm} (\beta; c , \boldsymbol{a}) = \int_{- \beta Y^{1+\vepsilon}}^{ \beta Y^{1+\vepsilon} }  \xi_{\pm} (r) \mathop{\sum \sum}_{m,n \sim Y}  \frac{a_m \overline{a}_n }{\sqrt{\beta} \sqrt[4]{mn} } S_{\bfF} (\vnu m, \vnu n ;c)   (\beta \sqrt{mn})^{i r} \nd r,
	\end{equation*}
	where on the integral domain $\xi_{\pm} (r) = O (1/\sqrt{\beta Y})$. Opening the Kloosterman sum as in \eqref{2eq: Kloosterman}, it follows from Cauchy that
	\begin{equation*}
		S_\vnu (\beta; c , \boldsymbol{a}) \Lt \frac 1 {\beta Y} \int_{- \beta Y^{1+\vepsilon}}^{ \beta Y^{1+\vepsilon} } \sumx_{\valpha (\mod\, c)} \bigg| \sum_{n \sim Y} a_n n^{ir/2} e_{\bfF}\Big[ \frac{\vnu n \valpha}{c}\Big] \bigg|^2   \nd r .
	\end{equation*}
	Finally an application of the large sieve inequality   \eqref{eq: cor: large sieve} in Corollary \ref{C cor: Gallagher thm 2} leads to   
	\begin{equation*}
		S_\vnu (\beta; c , \boldsymbol{a}) \Lt  \bigg( N(c) + \frac{r(c_{\vnu}) N((\vnu,c))}{\beta} \bigg) Y^\upvarepsilon \VA^2 . 
	\end{equation*}
%and hence \eqref{eq: lem: estimate of S_v beta} due to $\beta > 1/Y$. 
	For the easier case $\beta < 1/Y$, we use the Bessel integral representation %(alternatively, we may also use the Taylor expansion of $ J_0 (x) $)
	\begin{align*}%\label{2eq: Bessel's integral}
		J_{0}(x) = \frac{1}{\pi} \int_{0}^{\pi} \cos(  x \sin\theta) \nd \theta, 
	\end{align*}
to rewrite 
	\begin{equation*}
		S_\vnu (\beta; c , \boldsymbol{a}) = \frac{1}{2\pi} \int_{-\pi}^{\pi} \mathop{\sum \sum}_{m,n \sim Y}  a_m \overline{a}_n  S_{\bfF} (\vnu m, \vnu n ;c)  \exp(i \beta \sqrt{mn} \sin\theta) \nd \theta. 
	\end{equation*}
	Then we deduce from the same argument above by Lemma \ref{lem: Mellin of exp} and Corollary \ref{C cor: Gallagher thm 2} that  
	\begin{align*}
		S_\vnu (\beta; c , \boldsymbol{a}) & \Lt \int_{-Y^{\vepsilon}}^{Y^{\vepsilon}}  \sumx_{\valpha (\mod\, c)} \bigg| \sum_{n \sim Y} a_n n^{ir/2} e_{\bfF}\Big[ \frac{\vnu n \valpha}{c}\Big] \bigg|^2   \nd r \\
		& \Lt \big(   N(c) + Y r(c_{\vnu}) N((\vnu,c)) \big) Y^{\vepsilon} \VA^2, 
	\end{align*}
as desired. 
\end{proof}

\begin{coro}\label{cor: S(t)}
Let $ t \Lt 1$ be small.	Then
	 \begin{align}\label{7eq: S(t)}
	 	S_{\vnu} ( t ; \boldsymbol{a}; C) \Lt_{\vepsilon}  \bigg( \frac {C} {H }   + \sqrt{C} + \frac {  \sqrt{C}} {\sqrt{t^2 N(\vnu)}} \bigg) Y^\upvarepsilon \VA^2 . 
	 \end{align}
% for any $ C \Lt N(\vnu) Y^2/K^2 $. 
\end{coro}

\begin{proof}
As in the definition of $S_{\vnu} ( t ; \boldsymbol{a}; C)$ in \eqref{7eq: S (C)}, 
\begin{align*}
	\beta = 4\pi \omega  (t, \vnu / c , \vnu' /c'   ). 
\end{align*} Recall that
\begin{align*}
		\omega (t; x, y) = {  \sqrt{ (x-y)^2 + 4 xy \sin^2 (t/2)}},
\end{align*}
so  
\begin{align*}
	\frac 1 {\beta} \Lt \frac 1 {\sqrt{t^2 N (\vnu/c)}} .  
\end{align*}	By \eqref{7eq: S (C)} and \eqref{eq: lem: estimate of S_v beta} in Lemma \ref{lem: estimate of S_v beta}, 
	\begin{align*}
		S_{\vnu} (t ; \boldsymbol{a}; C) & \Lt Y^\upvarepsilon \VA^2 \sum_{ c\,\in \BSD_\vnu^{1}(C)  } \frac 1 {N(c)} \bigg(  N(c) + \frac { r(c_{\vnu}) N((\vnu,c))} {\sqrt{t^2 N (\vnu/c)}   }   \bigg) \\
		&  \Lt Y^\upvarepsilon \VA^2 \Bigg( \sum_{ c\,\in \BSD_\vnu^{1}(C)  }   1 +   \frac {  1 } { \sqrt{ t^2 C N(\vnu)   }} \mathop{\sum_{(\delta) \supset (\vnu)}}_{N (\delta) \Lt C}  N (\delta) \sum_{ c\,\in \BSD_{\vnu/\delta}^{1}(C / N (\delta) )  }  {r(c)  } \Bigg)  . 
	\end{align*}
The first sum yields  $ O  \big ((C   / H +\sqrt{C}) Y^\upvarepsilon \VA^2 \big)$ by   Corollary \ref{cor:lattice_count}. 
Note that  
\begin{align*}
	\sum_{ c\,\in \BSD_\mu^{1}( A )  }  {r(c)  } \Lt \sum_{ r \Lt \sqrt{A} } r \cdot  \text{\small \bf \#} \BSD_\mu^{1}( A/ r^2 )   \Lt \sum_{ r \Lt \sqrt{A} } \bigg(\frac {H^{\vepsilon} A} {H r}   + \sqrt{A}  \bigg) \Lt A ,  
\end{align*}
by   Corollary \ref{cor:lattice_count}, so the second sum yields $O \big( \sqrt{C / t^2 N(\vnu)} Y^{ \vepsilon} \VA^2 \big)$. %Thus the bound in \eqref{7eq: S(t)} follows as $C \Lt N(\vnu) Y^2/K^2$. 
	\end{proof}

	\subsection{Conclusion}
In view of \eqref{eq: h hat h} and \eqref{7eq: integral representation of P_1}, 
\begin{align*}
	P_{\vnu}^{+}(\boldsymbol{a};C) \Lt   {H}   \int_{- H^{\vepsilon} / H}^{H^{\vepsilon} / H}  \big|	S_{\vnu} (t ; \boldsymbol{a}; C) \big| \nd t . 
\end{align*}
As a direct consequence,  Corollary \ref{cor: S(t)} yields  
\begin{align*}
	P_{\vnu}^{+}(\boldsymbol{a};C) \Lt \bigg( \frac {C} {H }   + \sqrt{C} + \frac { H \sqrt{C}} {\sqrt{  N(\vnu)}}   \bigg) Y^\upvarepsilon \VA^2, 
\end{align*}
and hence the bound for $P_\vnu^1(\boldsymbol{a})$ as in \eqref{eq: estimate for P_v^1} because $ C \Lt  N(\vnu) Y^2/K^2$.

	\section{\texorpdfstring{Estimates for $P_\vnu(\boldsymbol{a};\varDelta)$}{Estimates for P\unichar{"03BD}(a;\unichar{"0394})}  }\label{sec: Estimate for P_Delta}
	
	\subsection{Application of Olver's Asymptotic}
	In view of  \eqref{2eq: gamma}, \eqref{Beq: Olver}, \eqref{3eq: P(a)}, \eqref{3eq: D(C)}, \eqref{3eq: D Delta}, \eqref{3eq: D1(C), D(C, Delta)}, \eqref{eq: def of B}, and Lemma \ref{lem: B(x, y), ranges} (2), up to a negligible error, we rewrite 
	\begin{equation}\label{6eq: P_Delta after Olver}
		P_\vnu(\boldsymbol{a}; C, \varDelta) = \sum_{2\, \nmid \, \vkappa} h(\vkappa)  \sum_{ c\,\in \BSD_\vnu (C, \varDelta) / \RU^2} \frac{ S_{\vnu, \vkappa} (c ;\boldsymbol{a}) }{N(c)}   , 
	\end{equation}
where $S_{\vnu, \vkappa}(c ;\boldsymbol{a})$ is the sum of four similar bilinear forms $S_{\vnu, \vkappa}^{\pm \pm} (c;\boldsymbol{a})$, two of which read 
\begin{align}\label{8eq: P+-}
	S_{\vnu, \vkappa}^{+ \pm } (c;\boldsymbol{a}) = \mathop{\sum \sum}_{m,n \sim Y} \! a_m \overline{a}_n S_{\bfF} (\vnu m, \vnu n ;c) V_\vkappa^{+\pm} \bigg( \frac{4\pi \vnu \sqrt{m n} }{c } \bigg) , 
\end{align}
with $\vkappa$ close to $K$, 
\begin{align}
	 V_\vkappa^{ + \pm} (x) = \exp( i\gamma_{\vkappa}( x) \pm i\gamma_{\vkappa}(\theta x)) \varvv_{ \pm}  (x; \theta ), \qquad \varvv_{ \pm}  (x; \theta ) = \frac{V_{+}(  x / \vkappa) V_{\pm}(\theta x/\vkappa )}{\sqrt{\theta}x},
\end{align}
and $\gamma_{\vkappa}(x)$   given by \eqref{6eq: def of gamma_kappa}.  
Recall from \eqref{3eq: theta} and \eqref{3eq: D Delta} that 
\begin{align*}
	\theta =   {\vnu' c  } / {\vnu c '  }, \qquad 1 -\theta  \sim \varDelta / H.
\end{align*}
As $\theta $ is near $1$, it follows from \eqref{2eq: V+-} that 
\begin{align}\label{8eq: bounds for v(x)}
	x^j \varvv_{\pm}^{(j)} (x; \theta ) \Lt_{j} \frac 1 {x} . 
\end{align}
It is crucial that Lemma \ref{lem: B(x, y), ranges} (2) provide us   the restriction
\begin{align}\label{8eq: range of C}
	 C \Lt \frac{N(\vnu) Y^{2} H^{\vepsilon} }{ \varDelta^2 K^2}.  
\end{align}

	\subsection{Application of the Hybrid Large Sieve}
In view of \eqref{8eq: P+-}--\eqref{8eq: range of C},   we may now apply Lemma \ref{lem: Mellin of w_kappa} with $X = 4\pi \vnu Y / c \asymp \sqrt{N(\vnu/ c)} Y \Gt  \varDelta K/H^\upvarepsilon$, $\delta = \varDelta / H$,  and $S=1/X$, open the Kloosterman sum, and use the Cauchy inequality to prove that
\begin{equation*}
	S_{\vnu, \vkappa}^{ + \pm} (c;\boldsymbol{a})   \Lt \int_{- X_{\pm} }^{X_{\pm}} |\xi^{\pm}(r)| \  \sumx_{ \valpha (\mod\, c) } \bigg| \sum_{n\sim Y} a_n n^{i r/2} e_{\bfF}\Big[ \frac{ \vnu n \valpha}{c}\Big] \bigg|^2 \nd r, 
\end{equation*}
up to a negligible error, where 
\begin{align*}
	X_+ = X^{1+\vepsilon}, \qquad X_- = \frac {\varDelta X^{1+\vepsilon} } H, 
\end{align*}
and 
\begin{align*}
	\xi_{+} (r) \Lt \frac 1 {X \sqrt{X}} , \qquad \xi_{-} (r) \Lt \frac 1 {X } . 
\end{align*}
By  the large sieve inequality   \eqref{eq: cor: large sieve} in Corollary \ref{C cor: Gallagher thm 2}, we have 
\begin{align*}
	S_{\vnu, \vkappa}^{+ +} (c;\boldsymbol{a})  \Lt   \frac {X N(c) + Y r (c_{\vnu}) N((\vnu, c))} {X \sqrt{X}}   Y^\upvarepsilon \VA^2  ,
\end{align*}
\begin{align*}
	S_{\vnu, \vkappa}^{ + - } (c;\boldsymbol{a})  \Lt   \frac {\varDelta X N(c)/ H  + Y r (c_{\vnu}) N((\vnu, c))} {X  }   Y^\upvarepsilon \VA^2 , 
\end{align*}
and hence
\begin{align}\label{8eq: P(c; a)}
	 S_{\vnu, \vkappa}^{+ +} (c;\boldsymbol{a}) + S_{\vnu, \vkappa}^{ + -} (c;\boldsymbol{a})  \Lt \! \bigg( \! \bigg(\frac{1}{\sqrt{\varDelta K}} + \frac{\varDelta  }{H } \bigg) N(c) + \frac{Y r(c_{\vnu}) N((\vnu, c)) }{ \varDelta K } \bigg) Y^{\vepsilon} \VA^2. 
\end{align}

\subsection{Conclusion} Next we borrow some arguments from the proof of Corollary \ref{cor: S(t)}.  Combining \eqref{6eq: P_Delta after Olver} and \eqref{8eq: P(c; a)}, we infer that
\begin{align*}
	 P_\vnu(\boldsymbol{a};C,\varDelta)   \Lt H Y^{\vepsilon} \VA^2 \! \sum_{ c\,\in \BSD_\vnu (C,\varDelta)} \frac 1 {N(c)} \bigg( \! \bigg(\frac{1}{\sqrt{\varDelta K}} + \frac{\varDelta  }{H } \bigg) N(c) + \frac{Y r(c_{\vnu}) N((\vnu, c)) }{ \varDelta K } \bigg) , 
\end{align*}
hence $P_\vnu(\boldsymbol{a};C,\varDelta)$ is bounded by the sum of 
\begin{align*}
	  Y^{\vepsilon} \VA^2   \bigg(\frac{H}{\sqrt{\varDelta K}} +  {\varDelta  }  \bigg) \sum_{ c\,\in \BSD_\vnu (C,\varDelta)}  1 ,
\end{align*}
and 
\begin{align*}
  Y^{\vepsilon} \VA^2 \cdot	 \frac {H Y} {C \varDelta K} \mathop{\sum_{(\delta) \supset (\vnu)}}_{N (\delta) \Lt C}  N (\delta) \sum_{ c\,\in \BSD_{\vnu/\delta}^{1}(C / N (\delta), \varDelta )  }  {r(c)  } . 
\end{align*}
	The former has bound   \begin{align*}
		  \bigg(\frac{H}{\sqrt{\varDelta K}} +  {\varDelta  }  \bigg) \bigg( \frac {C \varDelta}  H +\sqrt{C} \bigg) Y^\upvarepsilon \VA^2  , 
	\end{align*}  by   Corollary \ref{cor:lattice_count}. 
	Note that  
	\begin{align*}
		\sum_{ c\,\in \BSD_\mu^{1}( A , \varDelta)  }  {r(c)  } \Lt \sum_{ r \Lt \sqrt{A} } r \cdot  \text{\small \bf \#} \BSD_\mu^{1}( A/ r^2, \varDelta )   \Lt \sum_{ r \Lt \sqrt{A} } \bigg(\frac {\varDelta A} {H r}   + \sqrt{A}  \bigg) \Lt A \log A ,  
	\end{align*}
	by   Corollary \ref{cor:lattice_count}, so the latter  sum has bound 
\begin{align*}
	 \frac {H Y} {  \varDelta K} Y^\upvarepsilon \VA^2.  
\end{align*}
Since   $C \Lt {N(\vnu) Y^{2} H^{\vepsilon} } / { \varDelta^2 K^2}$  (see \eqref{8eq: range of C}), while $\varDelta$ ranges between  $H^{\vepsilon} $ and   $ H$, we conclude that 
\begin{align*}
	 P_\vnu(\boldsymbol{a};C,\varDelta)  \Lt \bigg(  \bigg(\frac{H}{\sqrt{  K}} +  1  \bigg) \bigg(\frac {Y^2} {HK^2} N(\vnu)+ \frac {Y} {K} \sqrt{N(\vnu)} \bigg) + \frac {HY} {K} \bigg) Y^\upvarepsilon \VA^2, 
\end{align*}
	as desired by \eqref{eq: estimate for P_Delta}.

		\section{Proof of Theorem \ref{thm: 2nd moment}}
	\subsection{Asai's $L$-functions}\label{sec: application}
	Recall that Asai's $L$-function $L (s, \mathrm{As}(f))$ is defined by
	\begin{equation*}
		L(s,\As(f) ) = \zeta(2s) \sum_{m \in \BZ_{+}}  \lambda_f(m) m^{-s}, \qquad \text{($\mathrm{Re} (s)>1$)}.
	\end{equation*}
	Asai \cite{Asai-1977} proved that $L(s,\As(f) ) $ has  analytic continuation to the whole $s$-plane with a possible simple pole at $s =  1$, and satisfies the functional equation
	\begin{equation*}
		\Lambda (s,\As(f) ) = \Lambda (1-s,\As(f) ),
	\end{equation*}
	where $\Lambda (s, \As(f) ) = d_{\bfF}^{s/2} \gamma (s, \As(f) ) L (s,\As(f) )$, with 
	\begin{equation*}
		 \gamma (s, \As(f) ) = (2\pi)^{-2s} \Gamma(s+k-1) \Gamma(s) .
	\end{equation*} 
Moreover, if we write
\begin{equation*}
	L(s,f) = \prod_{ \pi } \big( 1 - \valpha_f(\pi) N(\pi)^{-s}\big)^{-1} \big( 1 - \beta_f(\pi) N(\pi)^{-s}\big)^{-1},
\end{equation*}
where  $\pi $ denote primes in $ \SO^{+} /\RU^2 $, then
\begin{equation*}
	L (s,\As(f) ) =  \prod_p L_p(s,\As(f)), 
\end{equation*}
where
\begin{equation*}
	L_p^{-1}(s,\As(f))=
	\begin{cases}
		(1-\valpha_f(\pi_1)\valpha_f(\pi_2)p^{-s})(1-\valpha_f(\pi_1)\beta_f(\pi_2)p^{-s}) & \text{if } p \sim \pi_1\pi_2, \\
	\! 	\cdot (1-\beta_f(\pi_1)\valpha_f(\pi_2)p^{-s})(1-\beta_f(\pi_1)\beta_f(\pi_2)p^{-s}), & \ \ \, (\pi_1\not\sim \pi_2); \\
		(1-\valpha_f(\pi)p^{-s})(1-\beta_f(\pi)p^{-s})(1-p^{-2s}), & \text{if }  p \sim \pi; \\
		(1-\valpha_f^2(\pi)p^{-s})(1-\beta_f^2(\pi)p^{-s})(1-p^{-s}), & \text{if }  p \sim \pi^2. 
	\end{cases}
\end{equation*}
%For $h_{\bfF}^+ = 1$,  in the last case we have $ p = D $ and $\pi = \sqrt{D}$ (see \S \ref{sec: basic defn}). 

\subsection{Convoluted Asai $L$-functions} For $f \in S_k (\SL_2(\SO))$, assume that its Asai lift $\As(f)$ is cuspidal.  Let $ \phi \in S_{\frac 1 4 + \vvkappa^2} (\SL_2(\BZ)) $ be a Hecke--Maass cusp form of Laplace eigenvalue $1/4+\vvkappa^2$.  Let $\lambda_{\phi} (n)$ denote its Hecke eigenvalues.  The  analytic continuation and functional equation of $ L (s, \As(f)\times \phi) $ have been verified in  \cite{Ramakrishnan2002}. It is known that $ L (s, \As(f)\times \phi) $ is entire and satisfies
\begin{align}\label{8eq: FE} 
	\Lambda(s, \As(f)\times \phi) =   \Lambda(1-s, \As(f)\times \phi) , 
\end{align}
where  
the completed $L$-function $\Lambda(s, \As(f)\times \phi) = d_{\bfF}^{s} \gamma(s, \As(f)\times \phi) L(s, \As(f)\times \phi) $, and the  $\gamma$-factor 
\begin{align}\label{8eq: gamma}
	\gamma(s, \As(f)   \times \phi) \! = \! (2\pi)^{-4s}  \Gamma  (s +k-1-i\vvkappa  ) \Gamma  (s +k-1+i\vvkappa  ) \Gamma  (s  -i\vvkappa  ) \Gamma  (s  + i\vvkappa  )   .  
\end{align}
%Observe that \eqref{8eq: epsilon} and \eqref{8eq: gamma} indeed coincide with the formulae for  $ h \times h \times \phi   $ with $ h \in S_k (\SL_2(\BZ))$ and may be readily verified by the argument for $ \mathrm{Sym}^2 h \times \phi  $ as in \cite[\S  4.2.4]{Ho-Mu-Qi} (see Remark \ref{rem: QxQ}). 

\begin{remark}
	Note that if $\phi \in S_{2l} (\SL_2 (\BZ))$ were a holomorphic cusp form, then the root number 
	\begin{align*} 
		\vepsilon ( \As(f)\times \phi ) = \left\{ \begin{aligned}
			& -1, & & \text{ if } k > l, \\
		&  1, & & \text{ if } k \leqslant l, 
		\end{aligned} \right. 
	\end{align*} 
so we have trivially $ L (1/2, \As(f)\times \phi ) = 0 $ for large $k > l$. This is the reason why we consider here the case that $\phi$ is a Maass cusp form. 
\end{remark}

\begin{remark}
	We may as well consider the case that $\phi (z) = \partial E(z, s) /\partial s |_{s=\frac 1 2 }$ so that  $\lambda_{\phi} (n) = \tau (n)$, where $E (z, s)$ is the Eisenstein series for $\SL_2 (\BZ)$, but $ L ( 1/2,  \As(f)\times \phi  ) $ is not equal to $ |L (1/2, \As(f) )|^2 $. 
\end{remark}

For the application of the large sieve in Theorem \ref{thm: large sieve}, we must express  $	L (s, \As(f)\times \phi)$ in terms of Dirichlet series. %For simplicity, let us assume that $D$ is odd (hence square-free). For $ D $ even, we just need to replace $\sqrt{D}$ by $\sqrt{D}/ 2$ or $\sqrt{D} / \pi_2 $ ($2 \sim \pi_2^2 $) for $D \equiv 8$ or $12 \, (\mathrm{mod}\, 16)$ respectively  in \eqref{eq: lem: Dirichlet series of Rankin-Selberg}. 

	\begin{lem}\label{lem: Dirichlet series of Rankin-Selberg}
Let $\bfF = \bfQ (\sqrt{D})$ {\rm(}if $h^+_{\bfF} = 1$ then $D$ must be prime as in \S \ref{sec: basic defn}{\rm)}.   Then for $\mathrm{Re}(s) > 1$, we have
	\begin{equation}\label{eq: lem: Dirichlet series of Rankin-Selberg}
		L(s, \As(f)\times \phi) = \zeta(2s) Q(s, \phi ) \sum_{n \in \BZ_{+}} \  \mathop{   \sumn_{ \mu \, \in \SO^+ / \RU^2  }   }_{(\mu,D)= (1) }   \frac{\lambda_f(\mu^2 n \sqrt{D}) \lambda_\phi(n)}{ \lambda_f(\sqrt{D}) N(\mu)^{2s} n^s },
	\end{equation}
 	in which the superscript {\small $\natural$} indicates restriction to those $\mu$ with $r(\mu) = 1$ and 
  $Q(s, \phi) $ is the Euler product
	\begin{equation}
		Q (s, \phi ) =  \bigg( 1 - \frac{\lambda_\phi(D)}{D^s} + \frac{1}{D^{2s} } \bigg)^{-1} \cdot  \prod_{p\, \nmid D}\bigg( 1 - \frac{\lambda_\phi(p)^2 - 2}{p^{2s} } + \frac{1}{p^{4s} }\bigg)^{-1} .
	\end{equation}
\end{lem}

	The proof of Lemma \ref{eq: lem: Dirichlet series of Rankin-Selberg} outlined in Appendix \ref{Appendix: Rankin-Selberg} will recourse to a formal identity of Garrett, extracted from his work on triple product $L$-functions \cite{Garrett-Triple-Product}. 

Note that %$ \zeta (2s) Q(s, \phi) $ is an Euler product which converges absolutely for $\mathrm{Re}(s) > 1/2$. Moreover, 
 the double $(n, \mu)$-sum in \eqref{eq: lem: Dirichlet series of Rankin-Selberg} splits into 
\begin{align}\label{9eq: split n | D} 
	\sum_{n_1|D^\infty} \! \frac{\lambda_f(n_1 \sqrt{D} ) \lambda_\phi(n_1) }{\lambda_f(\sqrt{D})  n_1 ^s} \cdot \mathop{\sumn \sum}_{(\mu n,D)=(1)}  \frac{\lambda_f(\mu^2 n) \lambda_\phi(n)}{ N(\mu)^{2s} n^s } . 
\end{align}
We have the Deligne bound  (see \cite{Taylor-Hilbert}): 
\begin{equation*}%\label{9eq: Deligne}
	\lambda_f(\mu) \Lt_{\vepsilon} N(\mu)^{\vepsilon}. 
\end{equation*}
Note that by the Hecke relation, 
\begin{align*}
	 \frac{\lambda_f(n_1 \sqrt{D} )   }{\lambda_f(\sqrt{D}) } = \sum_{d|n_1} (-1)^{\mathrm{ord}_{D} (n_1/d)} \lambda_f (d)  , 
\end{align*}
%where $\Omega(n)$ is the number of prime factors of $n$ counted with multiplicity, 
and hence by the Deligne bound, 
\begin{align}\label{9eq: Deligne bound for n1}
	\frac{\lambda_f(n_1 \sqrt{D} )   }{\lambda_f(\sqrt{D}) } \Lt_{\vepsilon} n_1^{\vepsilon} .  
\end{align}
Moreover, we have the Kim--Sarnak bound \cite{Kim-Sarnak}
\begin{align}\label{9eq: Kim-Sarnak}
	 {\lambda_{\phi} (n )   }  \Lt_{\vepsilon} n^{7/64+\vepsilon} , 
\end{align}
and the averaged Ramanujan bound 
\begin{align}\label{9eq: Ramanujan}
\sum_{ n \leqslant X}	{|\lambda_{\phi} (n )|^2   }  \Lt_{\phi} X . 
\end{align}

\subsection{Set-up} It follows from \eqref{8eq: FE}, \eqref{eq: lem: Dirichlet series of Rankin-Selberg}, and \eqref{9eq: split n | D}  the approximate functional equation  as in \cite[Theorem 5.3]{IK}: 
\begin{align*}
	L(1/2 & ,   \As(f)\times \phi) = \\
	& \qquad 2 \sum_{n_1|D^\infty} \! \frac{\lambda_f(n_1 \sqrt{D} ) \lambda_\phi(n_1) }{\lambda_f(\sqrt{D})  \sqrt{n_1} }   \mathop{\sumn \sum}_{(\mu n,D)=(1)}  \frac{\lambda_f(\mu^2 n) \lambda_\phi(n)}{ N(\mu) \sqrt{n} } V_{k}  \bigg( \frac{ n n_1 N(\mu)^2}{ {d_{\bfF}} }\bigg) ,
\end{align*}
where 
\begin{align*}
	V_{k}  (y) = \frac{1}{2\pi i} \int_{(3)}  \zeta(1+2u) Q(1/2+u, \phi) \frac{\gamma(1/2+u, \As(f)\times\phi)}{\gamma(1/2, \As(f)   \times \phi)} \exp ({u^2})  y^{-u} \frac{\nd u}{u}. 
\end{align*}
By   \cite[Proposition 5.4]{IK}, we may restrict the summation to $ n n_1 N(\mu)^2 \Lt k^{2+\vepsilon} $ at the cost of a negligible error. Considering the square $ |L(1/2 ,   \As(f)\times \phi)|^2 $, we apply Cauchy to pull out the $ n_1 $- and $\mu$-sums so that 
\begin{equation}\label{9eq: after Cauchy}
\begin{split}
		\left| L( 1/2, \As(f)\times \phi ) \right|^2  & \Lt  k^\upvarepsilon  \mathop{ \sum_{n_1|D^\infty }  \,  \sumn_{ (\mu, D) = (1) } }_{n_1 N(\mu)^2 \Lt k^{2+\vepsilon}} \frac{1}{N(\mu)} \\
	 & \cdot   \bigg|   \mathop{\sum_{(n,D)=1}}_{n \Lt k^{2+\vepsilon} / n_1 N(\mu)^2 }  \frac{\lambda_f(\mu^2 n ) \lambda_\phi(n)}{ \sqrt{n} } V_{k} \bigg( \frac{ n_1 n N(\mu)^2}{ {d_{\bfF}} }\bigg)  \bigg|^2   ,
\end{split}
\end{equation}
where the bounds in  \eqref{9eq: Deligne bound for n1} and \eqref{9eq: Kim-Sarnak} have been implicitly applied.  

	\subsection{Application of Theorem \ref{thm: large sieve}}
Set $\vnu = \mu^2 $ and $Y = k^{2+\upvarepsilon}/n_1 N(\mu)^2$ in the notation of  \S \ref{secL intro}. Consider  the  average of \eqref{9eq: after Cauchy} over $f \in H_k$. On the right-hand side, the contribution in the case $ n_1 N(\mu)^2 \Gt k $ may be readily estimated by Luo's large sieve inequality \eqref{1eq: Luo LS, Asai} followed by %the Deligne bound \eqref{9eq: Deligne} and 
the averaged Ramanujan bound \eqref{9eq: Ramanujan}, yielding  
	\begin{align*}
		k^\upvarepsilon \mathop{\sum_{n_1|D^\infty } \sum_{ \mu } }_{ k \Lt \, n_1 N(\mu)^2 \Lt k^{2+\vepsilon} } \frac 1 {N(\mu)} \bigg(k^2 + \frac {k^4} {n_1^2 N (\mu)^2}   \bigg) \Lt k^{3+\vepsilon} ,
	\end{align*} 
while that in the case $ n_1 N(\mu)^2 \Lt k $ may be estimated in the same way by our large sieve inequality \eqref{1eq: main LS} in Theorem \ref{thm: large sieve}, yielding  
	\begin{equation*}
  k^\upvarepsilon \mathop{\sum_{n_1|D^\infty } \sum_{ \mu } }_{\, n_1 N(\mu)^2 \Lt k  }  \frac 1 {N(\mu)} \cdot \frac{k^{7/2} }{n_1^{3/2}  N(\mu) }  \Lt  k^{7/2+\upvarepsilon}.
\end{equation*}
At any rate, we have 
	\begin{equation*}%\label{1eq: 2nd moment bound}
	\sumx_{f\in H_k}  | L  (1/2, \As(f)\times  \phi  )  |^2 \Lt  k^{7/2+\upvarepsilon} . 
\end{equation*}
%This completes the proof of Theorem \ref{thm: 2nd moment}.

%Now we apply our large sieve inequality \eqref{1eq: main LS} in Theorem \ref{thm: large sieve} to the  average of \eqref{9eq: after Cauchy} over $f \in H_k$, choosing $\vnu = \mu^2 $, $Y = k^{2+\upvarepsilon}/n_1 N(\mu)^2$, %  (use Luo's \eqref{1eq: Luo LS, Asai} if $Y \leqslant k$), 	and the sequence $\boldsymbol{a} $ to be supported on those $n \in \BZ_+$ with $(n, D) = 1$. Consequently, along with the Deligne bound \eqref{9eq: Deligne}, we deduce that  

	\begin{appendices}
		
		\section{Stationary Phase}
	We record here Lemmas A.1 and A.2 in \cite{AHLQ-Bessel}, as variants of  \cite[Lemma {\rm 8.1}]{BKY-Mass} and \cite[Lemma 5.1.3]{Huxley}.   
	
	\begin{lem}\label{lem: staionary phase, dim 1}
		Let $\varww   \in C_c^{\infty} (a, b)$. Let  $f  \in C^{\infty} [a, b]$ be real-valued.  Suppose that there
		are parameters $P,\, Q,\, R,\, S,\, Z  > 0$ such that
		\begin{align*}
			f^{(i)} (x) \Lt_{ \, i } Z / Q^{i}, \qquad \varww^{(j)} (x) \Lt_{ \, j } S / P^{j},
		\end{align*}
		for  $i \geqslant 2$ and $j \geqslant 0$, and
		\begin{align*}
			| f' (x) | \Gt R. 
		\end{align*}
		Then 
		\begin{align*}
			\int_a^b  e (f(x)) \varww (x)  \nd x \Lt_{A} (b - a) S \bigg( \frac {Z} {R^2Q^2} + \frac 1 {R Q} + \frac 1 {R P} \bigg)^A  
		\end{align*} 
		for any  $A > 0$.
	\end{lem}
	 
\begin{lem}\label{lem: 2nd derivative}
	Let $\varww   \in C_c^{\infty} [a, b]$ and $V$ be its total variation. 	Let $f \in C^{\infty} [a, b]$ be real-valued. If $f'' (x) \geqslant \lambda > 0$ on $[a, b]$, then 
	\begin{align*}
		\bigg|\int_a^b  e (f(x)) \varww (x)  \nd x  \bigg| \leqslant \frac {4 V} {\sqrt{\pi \lambda}} . 
	\end{align*}
\end{lem}

\section{Olver's Uniform Asymptotic Formula}  \label{app: Olver}

Let $x > 1$.	According to the works of Olver \cite{{Olver-1},Olver-Bessel}, we have
\begin{equation}\label{12eq: Jm(mx) = }
	\begin{split}
		J_\vkappa (\vkappa x) = \Big(\frac {4\zeta} {1-x^2} \Big)^{1/4} \Bigg\{   \frac {\mathrm{Ai} \big(\vkappa^{2/3} \zeta\big)} {\vkappa^{1/3}} \sum_{s=0}^m \frac {A_s(\zeta)} {\vkappa^{2s}}   +   \frac { \mathrm{Ai}' \big(\vkappa^{2/3} \zeta\big) } {\vkappa^{5/3}} \sum_{s=0}^{m-1} \frac {B_s(\zeta)} {\vkappa^{2s}} & \\
		+ O \lp \frac{ 1 }{\vkappa^{2m+1} (1 + \vkappa^{1/6} |\zeta|^{1/4})  } \rp & \Bigg\},
	\end{split}
\end{equation}
where $\mathrm{Ai} (y)$ is the Airy function, 
\begin{equation} \label{12eq: gamma}
	\begin{aligned}
	%	&\frac 2 3 \zeta^{3/2}   = \log \frac {1 + \sqrt{1-x^2}} x - \sqrt{1 - x^2},  \quad & &  0 < x \leqslant 1, \\
		  \frac 2 3 (-\zeta)^{3/2}   = \int_1^x   {\sqrt{t^2-1}} \frac {\nd t} {t} =  \sqrt{x^2-1} - \mathrm{arcsec} \, x,   
	\end{aligned}
\end{equation}
and
\begin{align}\label{12eq: As Bs}
	A_s (\zeta) = \sum_{j=0}^{2s} b_j \zeta^{-3j/2} U_{2s-j} (v), \quad \zeta^{1/2} B_s (\zeta) = - \sum_{j=0}^{2s+1}   a_j \zeta^{-3j/2} U_{2s-j+1} (v),
\end{align}
in which $a_0 = b_0 = 1$,  
\begin{align*}
	a_s = \frac {1} {3^{2s} (2s)! } \frac { \Gamma (3s+1/2 )} {\Gamma (1/2)}, \qquad b_s = - \frac {6s+1} {6s-1} a_s,
\end{align*}
and $U_s (v)$ are polynomials in $v = 1 / \sqrt{1-x^2}$, with, for example, the first three be
\begin{align*}%\label{12eq: Us}
	U_0 = 1, \quad U_1 = (3v -5v^3)/24, \quad U_2 = (81v^2-462v^4+385v^6)/1152; 
\end{align*}
see  \cite[\S  4]{Olver-Bessel}  and \cite[Theorem B]{Olver-1} for the expansion and the error term as in \eqref{12eq: Jm(mx) = }, and \cite[\S  6]{Olver-Bessel} for the coefficients $A_s (\zeta)$ and $B_s (\zeta)$ as in \eqref{12eq: As Bs}. % (see also (4.5) in \cite{Olver-1}). 
As for the Airy function, if we set \begin{align*}
	 \gamma =  \frac 2 3 y^{3/2},  
\end{align*} then it is well-known (see  \cite[(10.4.15), (10.4.17)]{A-S}) that
\begin{equation*}
	\begin{split}
	%	& \mathrm{Ai} (y) = \frac {\sqrt{y } } {\sqrt 3 \pi}    K_{1/3} (\gamma), \qquad \quad 
	\mathrm{Ai} (-y) = \frac {\sqrt y } 3 \big(J_{1/3} (\gamma) + J_{-1/3} (\gamma)\big),  
	%	& \mathrm{Ai}' (y) = - \frac { {y } } {\sqrt 3 \pi}    K_{2/3} (\gamma), 
	\qquad \mathrm{Ai}' (-y) = \frac {  y } 3 \big(J_{2/3} (\gamma) -  J_{-2/3} (\gamma)\big).
	\end{split}  
\end{equation*} 
%For $|y| \Lt 1$, we have $\mathrm{Ai} (y) = O (1)$ (see \cite[(10.4.2)]{A-S}). 
For $y \Gt 1$, it follows from \cite[\S \S 7.21, 7.23]{Watson} that 
%\begin{align}\label{12eq: Ai, K}
%	\mathrm{Ai} (y) = O \bigg( \frac { \exp (- \gamma) } {  y^{1/4} } \bigg), \qquad \mathrm{Ai}' (y) = O \big( y^{1/4} { \exp (- \gamma) }  \big), 
%\end{align} 
%and
\begin{align}\label{12eq: Ai, J}
	\mathrm{Ai} (- y) = \frac 1 {\sqrt[4]{y}}  \big(  { \exp (   i \gamma) }  W_{+} (\gamma) + { \exp ( - i \gamma) }  W_{-} (\gamma) \big), 
\end{align}
\begin{align}\label{12eq: Ai, J, 2}
 \mathrm{Ai}' (- y) = \sqrt[4]{y}  \big(  { \exp (   i \gamma) }  W^{\flat}_{+} (\gamma) +   { \exp ( - i \gamma) }  W^{\flat}_{-} (\gamma)\big), 
\end{align}
with $ \gamma^j W_{\pm}^{(j)} (\gamma) \Lt_{j } 1 $ and $ \gamma^j W_{\pm}^{\flat \, (j)} (\gamma) \Lt_{j } 1 $; more explicitly,
\begin{align*}
	W^{\flat}_{\pm} (\gamma) = \bigg( \frac 1 {6\gamma} \mp i \bigg) W_{\pm} (\gamma) -  W'_{\pm} (\gamma) .
\end{align*}

\subsection{Proof of Lemma \ref{lem: Olver}} 
Keep in mind that  
\begin{align*}
	   \gamma / \vkappa =   \sqrt{x^2-1} - \mathrm{arcsec}\, x, \qquad i v = \frac 1 {\sqrt{x^2 - 1}}. 
\end{align*} For $x > 2$, it is clear that
\begin{align*}
 \gamma / \vkappa \asymp  x, \quad i v \asymp \frac 1 { x}, \qquad  x^j \frac {\nd^j  \gamma }  {\nd x^j } \Lt_j  \vkappa x, \quad 	x^j \frac {\nd^j v  }  {\nd x^j  }  \Lt_j \frac 1 {x}   . 
\end{align*}
Then Lemma \ref{lem: Olver} is a direct consequence of \eqref{12eq: Jm(mx) = }--\eqref{12eq: Ai, J, 2}.

		\section{Variant of Gallagher's Hybrid Large Sieve}\label{appendix: large sieve}
		
		In this appendix, we prove the hybrid large sieve in Lemma  \ref{C lem: Gallagher thm 2}. Let us start with Theorem 1 in \cite{Gallagher-LS}:
		\begin{lem}[Gallagher] 
		 We have
			\begin{equation}\label{Ceq: Gallagher thm 1}
				\int_{-T}^{T}  \Big| \sum a_n n^{it} \Big|^2 \nd t \Lt  T^2 \int_{0}^{\infty }  \bigg| \sum_{y  }^{\tau y}  a_n \bigg|^2 \frac{\nd y}{y}, \qquad \tau = e^{1/T}. 
			\end{equation}
		\end{lem}

For brevity, we write $ n \sim_{\tau} y $ for $ n$ in the segment  $ [y, \tau y] $ as above in \eqref{Ceq: Gallagher thm 1}.

Returning to the setting of Lemma \ref{C lem: Gallagher thm 2}, it follows from  \eqref{Ceq: Gallagher thm 1} that 
\begin{equation}\label{Ceq: applying lem 1}
	\sum_{\valpha (\mod\, c)} \int_{-T}^{T}  \bigg| \sum  a_n  n^{it} e_{\bfF} \Big[ \frac{n \valpha}{c}\Big] \bigg|^2 \nd t \Lt  T^2 \sum_{\valpha (\mod\, c)} \int_{0}^{\infty }  \bigg| \sum_{n \sim_{\tau} y} a_n e_{\bfF} \Big[ \frac{n \valpha}{c} \Big] \bigg|^2 \frac{\nd y}{y} . 
\end{equation}
By the orthogonality relation and the Cauchy inequality,
\begin{align*}
\begin{aligned}
		\sum_{\valpha (\mod\, c)}  \bigg| \sum_{n \sim_{\tau} y } a_n e_{\bfF}  \Big[\frac{n \valpha}{c} \Big] \bigg|^2 & = N(c) \sum_{\beta (\mod\, c)} \bigg| \mathop{\sum_{n \sim_{\tau} y }}_{n \equiv \beta (\mod \, c) } a_n   \bigg|^2 \\
		& \leqslant N(c) \sum_{\beta (\mod\, c)} \bigg(1 + \frac{\tau y - y}{N(c)/r(c)} \bigg)  \mathop{\sum_{n \sim_{\tau} y }}_{n \equiv \beta (\mod \, c) } |a_n|^2   , 
\end{aligned}
\end{align*} 
where we have used the simple  fact that, for two rational $m_1, m_2 \in \BZ$, the congruence $ m_1 \equiv m_2 (\mod \, c) $ over $\SO$ implies the congruence $ m_1 \equiv m_2 (\mod \, N(c)/r(c)) $ over $\BZ$ (recall that $r(c)$ is the largest rational divisor of $c$). Consequently,
\begin{align}
	\label{Ceq: bound}
	\sum_{\valpha (\mod\, c)}  \bigg| \sum_{n \sim_{\tau} y } a_n e_{\bfF}  \Big[\frac{n \valpha}{c} \Big] \bigg|^2  \leqslant  \big(N(c) +  { r(c)}  { ( \tau - 1) y }  \big)   {\sum_{n \sim_{\tau} y }}  |a_n|^2   . 
\end{align} 
Combining \eqref{Ceq: applying lem 1} and \eqref{Ceq: bound}, we infer that 
\begin{align*}
	\sum_{\valpha (\mod\, c)} \int_{-T}^{T}  \bigg| \sum  a_n  n^{it} e_{\bfF} \Big[ \frac{n \valpha}{c}\Big] \bigg|^2 \nd t \Lt  T^2   \int_{0}^{\infty }  \big(N(c) + { r(c)} { ( \tau - 1)y }   \big)  \sum_{ n \sim_{\tau} y } |a_n|^2  \frac{\nd y}{y}. 
\end{align*}
The coefficient of  $|a_n|^2$ here is  
\begin{align*}
	T^2 {N(c)} \int_{n/\tau}^{n}  \frac {\nd y} {y}  + T^2  r(c) (\tau-1) \int_{n/\tau}^{n}   {\nd y} = T N(c)   +   T^2  r(c) n \frac{(\tau - 1)^2  } {\tau}, 
\end{align*}
which is bounded by $  T N(c) + n r(c)$  as desired.

\section{Dirichlet Series for Convoluted Asai $L$-functions}\label{Appendix: Rankin-Selberg}
The purpose of this appendix is to prove the formula in Lemma \ref{lem: Dirichlet series of Rankin-Selberg}. It is reduced to the following formulae for the local factors of $ L  (s,\As(f)\times \phi) $.

\begin{lem}\label{lemD: Euler product of Rankin-Selberg}
	
Let $\mathrm{Re}(s) > 1$. 	Split %$L (s,\As(f)\times \phi) $ into 
	\begin{align*}
		L(s,\As(f)\times \phi) = \prod_p L_p(s,\As(f)\times \phi).
	\end{align*}
	Write  
\begin{align*}
\zeta_p (s) = \bigg(1 - \frac 1 {p^s} \bigg)^{-1},  \qquad 	Q_p (s, \phi) & =\left\{ \begin{aligned}
	\displaystyle &	\bigg( 1 - \frac{\lambda_\phi(p)^2 - 2}{p^{2s} } + \frac{1}{p^{4s} }\bigg)^{-1}, & & \text{ if } p \nmid \! D, \\
	\displaystyle   &   \bigg( 1 - \frac{\lambda_\phi(p)}{p^s} + \frac{1}{p^{2s} } \bigg)^{-1},  & & \text{ if } p \, |  D.
	\end{aligned} \right. 
\end{align*}

	{\rm (1)} If $p \sim \pi_1 \pi_2$ {\rm(}$\pi_1 \not \sim \pi_2${\rm)}, then
	\begin{equation*}
		\begin{split}
			L_p(s,\As(f)\times \phi)   = \zeta_p (2s) Q_p (s, \phi) \mathop{\sum \sum \sum}_{ \min\{ v  , w  \} = 0 }   \frac {\lambda_f (p^j \pi_1^{2v }  \pi_2^{2 w }) \lambda_\phi(p^j)} {p^{(j+2 v +2 w )s}}.
		\end{split}
	\end{equation*}
	
	{\rm (2)} If $p \sim \pi$, then
	\begin{equation*}
		L_p(s,\As(f)\times \phi) = \zeta_p (2s) Q_p (s, \phi) \sum_{j = 0}^{\infty} \frac{\lambda_f(p^j) \lambda_\phi(p^j)}{ p^{js} }.
	\end{equation*}
	
	{\rm (3)} If $p \sim \pi^2$ {\rm(}indeed, $p = D$ for $h_{\bfF}^{+} = 1${\rm)}, then
	\begin{equation*}
		L_p(s,\As(f)\times \phi) = \zeta_p (2s) Q_p (s, \phi)  \sum_{j=0}^{\infty} \frac{ \lambda_f(p^j \pi) \lambda_\phi(p^j)}{ \lambda_f(\pi) p^{js} }.
	\end{equation*}
\end{lem}

Our  focus will primarily be on (1), for which our reference is Garrett's work on   $\GL_2\times \GL_2 \times \GL_2$ triple products  \cite[\S 4]{Garrett-Triple-Product}. The cases of (2) and (3) become simpler and are essentially parallel to that of $\GL_2\times \GL_2 $ Rankin--Selberg products. %Thus we shall omit easy calculations and repetitive arguments. Note that the ramified case when $p = \pi^2$ does not play a major role in our problem.

The next lemma is extracted from Garrett's arguments in  \cite[\S 4]{Garrett-Triple-Product}. 

\begin{lem}[Garrett]\label{Dlem: Garrett's identity}
	Write $a'=1/a,\, b'=1/b,\, c'=1/c$. Then the formal series
	\begin{equation}\label{Deq: indentity_left}
		\mathop{\sum \sum \sum}_{  \min\{ v,w\} = 0 }  \frac{a^{j+2v+1} - a^{\prime \, j+2v+1}  }{a - a'}\cdot\frac{b^{j+2w+1} - b^{\prime \, j+2w+1} }{b - b'}\cdot \frac{c^{j+1} - c^{\prime \, j+1} }{c - c'}\cdot X^{j+2v+2w}
	\end{equation}
	is equal to 
\begin{equation}\label{Deq: indentity_right}
	\begin{split}
\frac {(1-X^2) (1-c^2 X^2)   (1 - c^{\prime \, 2}  X^2)} { (1-abcX)(1-a'bcX) \cdots  (1-ab'c'X)(1-a'b'c'X)}{\rm;}
	\end{split}
\end{equation}
the denominator is symmetric with respect to $a,\, b,\, c$ {\rm(}so the four omitted terms involve $ ab'c , \,  a b c', \, a'b'c, \, a'bc'${\rm)}. 
\end{lem}

	In Garrett's setting, $a$, $b$, $ c$ are respectively taken as the Satake parameters of three cusp forms for $\mathrm{SL}_2 (\BZ)$. In  our  case,  we set $a = \valpha_f(\pi_1)$, $b = \valpha_f(\pi_2)$, $c = \valpha_\phi(p)$, and $X = p^{-s}$. Note that the denominator in \eqref{Deq: indentity_right} is just $L_p^{-1}(s,\As(f)\times \phi)$ for $p \sim \pi_1 \pi_2$. Thus Lemma \ref{lemD: Euler product of Rankin-Selberg} (1) follows readily from Lemma \ref{Dlem: Garrett's identity}, along with  
	\begin{align*}
		1 - (\lambda_\phi(p)^2 - 2)p^{-2s} + p^{-4s} = (1-c^2 X^2)(1 - c^{\prime \, 2} X^2), 
	\end{align*} 
\begin{align*}
  \lambda_f (\pi_1^m) = \frac{a^{m+1} - a^{\prime \, m+1} }{a - a'} , \quad \lambda_f (\pi_2^m) = \frac{b^{m+1} - b^{\prime \, m+1} }{b - b'} , \quad 	\lambda_\phi(p^m) = \frac{c^{m+1} - c^{\prime \, m+1}  }{c - c'} . 
\end{align*} 

\begin{lem}\label{Dlem: R-S}
	Write $a'=1/a $ and $b'=1/b$. Then
	\begin{equation}
		\sum_{ j=0 }^{\infty}\!  \frac{a^{j+1} - a^{\prime \, j+1}  }{a - a'} \cdot  \frac{b^{j+1} - b^{\prime \, j+1} }{b - b'} \cdot  X^{j} \! =\! \frac{1-X^2}{(1-abX)(1-a'bX)(1-ab'X)(1-a'b'X)}.
	\end{equation}
\end{lem}

Finally, Lemma  \ref{lemD: Euler product of Rankin-Selberg} (2), (3) follow directly from Lemma \ref{Dlem: R-S}. 
For the case $p \sim \pi$ or $ \pi^2$, we set $a = \valpha_f(\pi)$ or $ \valpha^2_f(\pi)$, $b = \valpha_\phi(p)$, and $X = p^{-s}$. Note that in both cases
$$ Q_p (s,\phi) L_p^{-1}(s,\As(f)\times \phi) = (1-abX)(1-a'bX)(1-ab'X)(1-a'b'X). $$

\subsection{Outline of the Proof of Lemma \ref{Dlem: Garrett's identity}}

For completeness, let us outline here Garrett's arguments in \cite[\S 4]{Garrett-Triple-Product}. 

As the identity is formal,   let us suppose that $a,\, b,\, c$ are algebraically independent over $\bfK = \overline{\BQ}(X)$. Let
\begin{align*}
	\mathrm{G} & = \mathrm{Gal}\big( \bfK(a,b,c) / \bfK(a+a',b+b',c+c') \big) .
	%	\mathrm{H} & = \mathrm{Gal}\big( \bfK(a,b,c) / \bfK(a+a',b+b',c) \big).
\end{align*} 
Clearly $\mathrm{G} \approx (\BZ/2\BZ)^3$. % and $\mathrm{H} \approx (\BZ/2\BZ)^2$. 
Define the sign homomorphism
\begin{align*}
	\mathrm{sgn}: (\BZ/2\BZ)^3 \ra \{ \pm 1 \}, \qquad \mathrm{sgn} (x_1,x_2,x_3) = \exp(\pi i (x_1 + x_2 + x_3)).  
\end{align*} 
Now  we can express \eqref{Deq: indentity_right} as
\begin{equation*}
	(1-X^2) (1-c^2 X^2) (1 - c^{\prime \, 2} X^2) \prod_{\upsigma \in \mathrm{G}} \frac 1  {1 - \upsigma (abc)X}  .
\end{equation*}
On the other hand, we may rewrite \eqref{Deq: indentity_left} as
\begin{align*}
	\frac{1}{(a-a')(b-b')(c-c')}  \mathop{\sum \sum \sum}_{  \min\{ v,w\} = 0 }\, \sum_{ \upsigma \in \mathrm{G} }\mathrm{sgn}(\upsigma) \upsigma \big(a^{j+2v+1} b^{j+2w+1} c^{j+1} \big) X^{j+2v+2w}.  
\end{align*}
Summing the geometric series, with the restriction $ \min\{ v,w\} = 0$, we transform the expression above into  
\begin{equation}\label{Deq: simplified series}
	\frac{1}{(a-a')(b-b')(c-c')} \! \sum_{ \upsigma \in \mathrm{G} }\! \mathrm{sgn}(\upsigma) \upsigma \left\{ \frac{abc}{1 - abcX}\bigg( \frac{1}{1-a^2 X^2} + \frac{1}{1-b^2 X^2} - 1  \bigg) \! \right\} \! .
\end{equation}
%	This simplifies to ??????
This is a rational function in $X$, which can be written as
\begin{align*}
	\frac{P(X) }{ (1-a^2 X^2) (1-a^{\prime \, 2} X^2) (1-b^2 X^2)(1-b^{\prime \, 2} X^2) \prod_{ \upsigma \in \mathrm{G} } \big( 1 - \upsigma (abc)X \big) } .
\end{align*}
%where $P(X)$ is a polynomial. 
Therefore it is left to prove  
\begin{align*}
	P(X) \! =\! (1-X^2)(1-a^2 X^2)(1-a^{\prime \, 2} X^2)(1-b^2 X^2)(1-b^{\prime \, 2} X^2)(1-c^2 X^2)(1-c^{\prime \, 2} X^2). 
\end{align*}
To this end, one just needs to verify that $P (X)$ has the indicated factors and highest-degree term $- X^{14}$, by calculating the sum or sub-sums in \eqref{Deq: simplified series} at special values like $X = \pm 1/ a, \cdots, \, \pm 1, \, \text{`$\infty$'}$. Such calculations have been done by Garrett in \cite[\S 4]{Garrett-Triple-Product} (or may be carried out by a software like Mathematica). 

Garrett remarked: ``One could, in principle, simply `guess' the factorization of the polynomial $P$ above, and check by multiplying out. However, the above symmetry arguments certainly enhance such `guessing'."

	\end{appendices}
%\bibliographystyle{alphanum}
%\bibliography{references}

\def\cprime{$'$}

\end{document}